\documentclass{article}
\usepackage{amssymb}
\usepackage{eurosym}
\usepackage{amsmath}
\usepackage{amsfonts}
\usepackage[ignoreall]{geometry}
\usepackage[title]{appendix}

\newtheorem{theorem}{Theorem}

\newtheorem{definition}[theorem]{Definition}

\newtheorem{lemma}[theorem]{Lemma}

\newtheorem{proposition}[theorem]{Proposition}
\newtheorem{remark}[theorem]{Remark}

\newenvironment{proof}[1][Proof]{\noindent\textbf{#1.} }{\ \rule{0.5em}{0.5em}}
\input{tcilatex}
\begin{document}

\title{Inequalities from Lorentz-Finsler norms}
\author{Nicu\c{s}or Minculete$^{1}$, Christian Pfeifer$^{2}$, Nicoleta Voicu$%
^{3}$ \\
%EndAName
$^{1,3}${\small Faculty of Mathematics and Computer Science,Transilvania
University of Brasov, }\\
{\small Iuliu Maniu 50, 500091 Brasov, Romania}\\
$^{2}${\small Laboratory of Theoretical Physics, Institute of Physics,
University of Tartu, }\\
{\small W. Ostwaldi 1, 50411 Tartu, Estonia}}
\date{}
\maketitle

\begin{abstract}
We show that Lorentz-Finsler geometry offers a powerful tool in obtaining
inequalities. With this aim, we first point out that a series of famous
inequalities such as: the (weighted) arithmetic-geometric mean inequality,
Acz\'{e}l's, Popoviciu's and Bellman's inequalities, are all particular
cases of a reverse Cauchy-Schwarz, respectively, of a reverse triangle
inequality holding in Lorentz-Finsler geometry. Then, we use the same method
to prove some completely new inequalities, including two refinements of Acz%
\'{e}l's inequality.
\end{abstract}

\section{Introduction}

The Cauchy-Schwarz inequality on the Euclidean space $\mathbb{R}^{n}:$ 
\begin{equation}
\left( \sum_{i=1}^{n}v_{i}^{2}\right) \cdot \left(
\sum_{i=1}^{n}w_{i}^{2}\right) \geq \left( \sum_{i=1}^{n}v_{i}w_{i}\right)
^{2},  \label{1_0}
\end{equation}%
$\forall v=\mathbf{(}v_{1},...,v_{n}\mathbf{),}$ $w=\mathbf{(}w_{1},...,w_{n}%
\mathbf{)\in }\mathbb{R}^{n}$, is a basic result, with applications in
almost all the branches of mathematics.

In 1956, Acz\'{e}l \cite{AC} introduced the following inequality 
\begin{equation}
\left( v_{0}^{2}-v_{1}^{2}-...-v_{n}^{2}\right) \left(
w_{0}^{2}-w_{1}^{2}-...-w_{n}^{2}\right) \leq \left(
v_{0}w_{0}-v_{1}w_{1}-...-v_{n}w_{n}\right) ^{2},  \label{1_1}
\end{equation}%
(holding for all $v\mathbf{=(}v_{0},v_{1},...,v_{n}\mathbf{),}$ $w\mathbf{=(}%
w_{0},w_{1},...,w_{n}\mathbf{)\in }\mathbb{R}^{n+1}$ such that $%
v_{0}^{2}-v_{1}^{2}-...-v_{n}^{2}>0,$ $w_{0}^{2}-w_{1}^{2}-...-w_{n}^{2}>0$%
), in relation to the theory of functional equations in one variable. The Acz%
\'{e}l inequality (\ref{1_1}), together with its generalization to
Lorentzian manifolds, known by the name of \textit{reverse Cauchy-Schwarz
inequality,} proved to be crucial to relativity theory and to theories of
physical fields.

Indeed, from a geometric standpoint, the two inequalities above are known to
be two sides of the same coin; while the usual Cauchy-Schwarz inequality (%
\ref{1_0}) is extended to positive definite inner product spaces - and
further on, to Riemannian manifolds - leading to the triangle inequality $%
\left\Vert v+w\right\Vert \leq \left\Vert v\right\Vert +\left\Vert
w\right\Vert ,$ (\ref{1_1}) is naturally extended to spaces with a
Lorentzian inner product (and more generally, to Lorentzian manifolds),
leading to a reverse triangle inequality (see, e.g., \cite{BeemEhrlich,ONeil}%
).

\bigskip

In this article we discuss a further generalization of the above picture
which has not been exploited so far. The Cauchy-Schwarz inequality and its
Lorentzian-reversed version can be extended to Finsler, \cite{Bao},
respectively, to Lorentz-Finsler spaces, \cite%
{Aazami:2014ata,Javaloyes2019,Minguzzi2014,Minguzzi:2013sxa}. Roughly
speaking, while a Riemannian manifold is a space equipped with a smoothly
varying family of inner products, a Finsler manifold is a space equipped
with a family of \textit{norms}\footnote{%
Actually, the notion of Finsler norm is slightly more general than the usual
one, as it is only required to be positively homogeneous instead of
absolutely homogeneous.} that do not necessarily arise from a scalar
product. Similarly, a Lorentz-Finsler manifold is equipped with a smoothly
varying family of so-called \textit{Lorentz-Finsler norms} of vectors, that
do not necessarily arise as the square root of any quadratic expression -
but are just positively 1-homogeneous in the considered vectors.

\bigskip

Usually, the Finslerian Cauchy-Schwarz inequality (also called the \textit{%
fundamental inequality, }\cite{Bao}) is proven under the assumption that the
Finsler norm $F$ of vectors has the property that the Hessian $Hess(F^{2})$
is positive definite; in the particular case of Riemannian spaces, this will
turn into the condition that the metric tensor $g$ is positive definite.
Similarly, its Lorentzian-reversed counterpart, \cite%
{Aazami:2014ata,Javaloyes2019,Minguzzi2014,Minguzzi:2013sxa} is proven under
the assumption that $Hess(F^{2})$ has Lorentzian signature for all vectors
in a strictly convex set. Under these assumptions, the obtained inequalities
are \textit{strict}, i.e., equality only holds when the vectors $v$ and $w$
are collinear.

\bigskip

As a preliminary step, we point out that the above conditions can be
relaxed. Namely, the respective inequalities still hold - just, non-strictly
- if we allow $Hess(F^{2})$ to be degenerate along some directions; also,
for practical matters, we replace the strict convexity assumption on the set
of interest with the more relaxed one that $F$ is defined on a convex conic
domain\textbf{\ }$\mathcal{T}$. An extension to the case when $F\ $is not
smooth is also presented. Moreover, we prove two refinements of the reverse
triangle inequality holding in general Lorentz-Finsler spaces in Section \ref%
{sec: refinements triangle}.

\bigskip

While the above generalizations are not spectacular for themselves, they
allow us much more freedom in choosing the range of examples and
applications. Indeed, we show in Section~\ref{sec:ex} that some of the most
famous inequalities on $\mathbb{R}^{n}$ are nothing but reverse
Cauchy-Schwarz inequalities for conveniently chosen (possibly, degenerate)
Lorentz-Finsler norms:

\begin{enumerate}
\item The usual arithmetic-geometric mean inequality: $\dfrac{1}{n}~\overset{%
n}{\underset{i=1}{\sum }}v_{i}\geq {\large (}\overset{n}{\underset{i=1}{%
\prod }}v_{i}{\large )}^{1/n},$ $\ \ \forall v_{i}\geq 0,$ $i=\overline{1,n}%
. $

\item The weighted arithmetic-geometric mean inequality: 
\begin{equation}
\dfrac{1}{a}\sum_{i=1}^{n}a_{i}v_{i}\geq \lbrack
(v_{1})^{a_{1}}...(v_{n})^{a_{n}}]^{1/a},  \label{weighted_a}
\end{equation}%
for all $a,a_{i},v_{i}\in \mathbb{R}_{+}^{\ast }$, such that $\overset{n}{%
\underset{i=1}{\sum }}a_{i}=a.$

\item Popoviciu's inequality, \cite{Pop}: 
\begin{equation}
\left( v_{0}^{p}-v_{1}^{p}-...-v_{n}^{p}\right) ^{1/p}\left(
w_{0}^{q}-w_{1}^{q}-...-w_{n}^{q}\right) ^{1/q}\leq
v_{0}w_{0}-v_{1}w_{1}-...-v_{n}w_{n},  \label{1_2}
\end{equation}%
holding for all $v_{i},w_{i}>0,$ such that $%
v_{0}^{p}-v_{1}^{p}-...-v_{n}^{p}>0$ and $%
w_{0}^{q}-w_{1}^{q}-...-w_{n}^{q}>0;$ the powers $p,q>1$ are such that $%
\dfrac{1}{p}+\dfrac{1}{q}=1$. In particular, for $p=q=\dfrac{1}{2},$
Popoviciu's inequality yields Acz\'{e}l's inequality.

\item Another result, due to Bellman \cite{Bel, Mitrinovic1970,
Mitrinovic1993}: 
\begin{equation}
\left( v_{0}^{p}-v_{1}^{p}-...-v_{n}^{p}\right) ^{1/p}+\left(
w_{0}^{p}-w_{1}^{p}-...-w_{n}^{p}\right) ^{1/p}\leq \lbrack
(v_{0}+w_{0})^{p}-(v_{1}+w_{1})^{p}-...-(v_{n}+w_{n})^{p}]^{1/p},
\label{Bellman}
\end{equation}%
(with $v_{i},w_{i}$ as above and$\ p>1$) is just the reverse triangle
inequality corresponding to (\ref{1_2}).
\end{enumerate}

Further, in Sections~\ref{Section_bimetric} and \ref{Section_Kropina}, we
use two particular classes of Lorentz-Finsler norms (the so-called \textit{%
bimetric} \cite{PerlickBook,Pfeifer:2011tk,Punzi:2007di} and $\left( \alpha
,\beta \right) $-metrics (with focus on a particular case, \textit{Kropina}
norms \cite{Kropina}), in order to prove some new inequalities.

\bigskip

In Section \ref{Aczel section}, we prove the following class of inequalities
on $\mathbb{R}^{n+1}$:%
\begin{equation}
\lbrack v_{0}w_{0}-\bar{g}_{\vec{v}}(\vec{v},\vec{w})]^{2}-[v_{0}^{2}-\left%
\Vert \vec{v}\right\Vert ^{2}][w_{0}^{2}-\left\Vert \vec{w}\right\Vert
^{2}]\geq 0  \label{Aczel_Finsler}
\end{equation}%
(for all $\vec{v}=\left( v_{1},....,v_{n}\right) ,\vec{w}=\left(
w_{1},....,w_{n}\right) \in \mathbb{R}^{n},$ $v_{0},w_{0}>0$ such that $%
v_{0}^{2}-\left\Vert \vec{v}\right\Vert ^{2}\geq 0,$ $w_{0}^{2}-\left\Vert 
\vec{w}\right\Vert ^{2}\geq 0$), where $\left\Vert \vec{v}\right\Vert =\bar{F%
}(\vec{v})$ is an arbitrary Finsler norm on $\mathbb{R}^{n}$ and $\bar{g}_{%
\vec{v}}=\dfrac{1}{2}Hess_{\vec{v}}(\bar{F}^{2})$ is the corresponding
Finsler metric tensor -- thus generalizing the usual Acz\'{e}l inequality.
Using the positive definite version of the Finslerian Cauchy-Schwarz
inequality, we then find two refinements thereof:%
\begin{equation}
\lbrack v_{0}w_{0}-\bar{g}_{\vec{v}}(\vec{v},\vec{w})]^{2}-[v_{0}^{2}-\left%
\Vert \vec{v}\right\Vert ^{2}][w_{0}^{2}-\left\Vert \vec{w}\right\Vert
^{2}]\geq \dfrac{\left( w^{0}\right) ^{2}-\left\Vert \vec{w}\right\Vert ^{2}%
}{\left\Vert \vec{w}\right\Vert ^{2}}\left( \left\Vert \vec{v}\right\Vert
^{2}\left\Vert \vec{w}\right\Vert ^{2}-\bar{g}_{\vec{v}}(\vec{v},\vec{w}%
)\right) ;  \label{refinement1}
\end{equation}

\begin{equation}
\lbrack v_{0}w_{0}-\bar{g}_{\vec{v}}(\vec{v},\vec{w})]^{2}-[v_{0}^{2}-\left%
\Vert \vec{v}\right\Vert ^{2}][w_{0}^{2}-\left\Vert \vec{w}\right\Vert
^{2}]\geq \left[ w^{0}\dfrac{\bar{g}_{\vec{v}}(\vec{v},\vec{w})}{\left\Vert 
\vec{w}\right\Vert }-v^{0}\left\Vert \vec{w}\right\Vert \right] ^{2}.
\label{refinement2}
\end{equation}

\section{Reverse inequalities for Lorentzian bilinear forms\label%
{sec:classineq}}

Before we study the extended Finslerian case, we briefly recall the
classical inequalities for bilinear forms.

Throughout the paper, we denote by $V$ a real $\left( n+1\right) $%
-dimensional space. We will use Einstein's summation convention, if not
otherwise explicitly stated: whenever in an expression an index $i$ appears
both as a superscript and as a subscript, we will automatically understand
summation over all possible values of $i$, i.e., instead of $\underset{i=0}{%
\overset{n}{\sum }}a_{i}b^{i},$ we will write simply, $a_{i}b^{i}$. This is
why, we will typically number components of vectors with superscripts from $%
0 $ to $n$, rather than with subscripts; this way, the expression of a
vector $v\in V$ in the basis $\left\{ e_{i}\right\} _{i=\overline{0,n}}$
will be written as%
\begin{equation*}
v=v^{i}e_{i}.
\end{equation*}%
Unless elsewhere specified, by "smooth", we will mean $\mathcal{C}^{\infty }$
(though usually, differentiability of some finite order is sufficient). We
will denote by $i,j,k...$\ indices running from $0$\ to $n$\ and by Greek
letters $\alpha ,\beta ,\gamma ,..,$\ indices running from $1$\ to $n.$

\bigskip

A \textit{Lorentzian scalar product, \cite{BeemEhrlich, Minguzzi2019, ONeil},%
} on the $\left( n+1\right) $-dimensional vector space $V$ is a symmetric
bilinear form $g:V\times V\rightarrow \mathbb{R}$ of index $n$. If $V$
admits a Lorentzian scalar product, it is called an $(n+1)$-dimensional 
\textit{Minkowski spacetime}. Choosing an arbitrary basis, we have: 
\begin{equation}
g(v,v)=g_{ij}v^{i}v^{j},  \label{Lorentzian bilinear form general}
\end{equation}%
where $(g_{ij})$ is a matrix with constant entries. In particular, in a $g$%
-orthonormal basis of $V,$ the bilinear form $g$ has the expression%
\begin{equation}
g(v,v)=\eta _{ij}v^{i}v^{j}=\left( v^{0}\right) ^{2}-\left( v^{1}\right)
^{2}-...-\left( v^{n}\right) ^{2},  \label{Minkowski metric}
\end{equation}%
where $\left( \eta _{ij}\right) =diag(1,-1,-1,...,-1).$

A nonzero vector $v\in V$ is called \textit{timelike} if $g(v,v)>0$ and 
\textit{causal}, if $g(v,v)\geq 0$. The set of causal vectors consists of
two connected components, corresponding to the choices $v^{0}>0$ and $%
v^{0}<0 $ respectively in a given (arbitrary) $g$-orthonormal basis.

In the following, we will denote by $C$ one of these two connected
components. By conveniently choosing the basis, we can assume that, for all $%
v\in C,$ we have $v^{0}>0.$ The elements of $C$ are called \textit{%
future-directed causal vectors.}

Denote:%
\begin{equation}
F(v):=\sqrt{g(v,v)},~\ \ \forall v\in C.  \label{pseudo-norm}
\end{equation}%
The function $F:C\rightarrow \mathbb{R}^{+}$ defined by the above relation
is sometimes called, by analogy with the Euclidean case, the Lorentzian
(pseudo-)\textit{norm }associated to the Lorentzian scalar-product $g$.

We will denote by:%
\begin{equation}
\mathcal{T}~\ \mathcal{=}\left\{ v\in C~\ |~F(v)>0\right\} ,
\label{definition_T}
\end{equation}%
the subset of $C$ consisting of timelike vectors. Elements of $\mathcal{T}$
are called \textit{future-directed timelike vectors}. The set $\mathcal{T}$
is always convex.

On a Minkowski spacetime $(V,g),$ the following inequalities hold (see,
e.g., \textit{\cite[Proposition 30]{ONeil}})

\begin{itemize}
\item \textbf{Reverse Cauchy-Schwarz inequality}:%
\begin{equation}
g(v,w)\geq F(v)F(w),~\ \ \ \forall v,w\in C;
\label{classical reverse CS ineq}
\end{equation}

\item \textbf{Reverse triangle inequality:}%
\begin{equation*}
F(v+w)\geq F(v)+F(w),~\ \forall v,w\in C.
\end{equation*}%
These inequalities are \textit{strict}, in the sense that equality holds if
and only if $v$ and $w$ are collinear.
\end{itemize}

\bigskip

\textbf{Particular case (Acz\'{e}l's inequality). }For $V=\mathbb{R}^{n+1}$
equipped with a $g$-orthonormal basis, the reverse Cauchy-Schwarz inequality 
\begin{equation}
(v^{0}w^{0}-v^{1}w^{1}-....-v^{n}w^{n})^{2}\geq \lbrack \left( v^{0}\right)
^{2}-\left( v^{1}\right) ^{2}...-\left( v^{n}\right) ^{2}][\left(
w^{0}\right) ^{2}-\left( w^{1}\right) ^{2}...-\left( w^{n}\right) ^{2}],
\label{Aczel}
\end{equation}%
$\forall v,w\in C$, becomes Acz\'{e}l's inequality (\ref{1_1}).

\bigskip

\begin{remark}
\textbf{(Positive definite bilinear form\textbf{s}):} In the case when the
metric $g$ is positive definite, we have: $C=V,$ $\mathcal{T}=V\backslash
\{0\}.$ The usual, non-reversed Cauchy-Schwarz inequality 
\begin{equation}
g(v,w)\leq F(v)F(w)  \label{non-reversed CS}
\end{equation}%
and the usual triangle inequality:%
\begin{equation}
F(v+w)\leq F(v)+F(w)  \label{non-reversed triangle ineq}
\end{equation}%
hold strictly on the entire space $V,$ see for example \cite[Proposition 18]%
{ONeil}.
\end{remark}

\section{\label{ineq_Finsler norms}Finsler and Lorentz-Finsler functions on
a vector space}

We saw in the previous section that the famous reverse and non-reverse
Cauchy-Schwarz inequalities, thus in particular Acz\'{e}l's inequality, are
closely connected to the geometric concepts of (pseudo)-Riemannian geometry.
Here we show that further famous inequalities are also related to a
geometric concept, namely to the concept of \textit{(pseudo)-Finsler geometry%
}.

\subsection{Finsler structures on a vector space}

Let $V$ be a real $\left( n+1\right) $-dimensional space as above.

A \textit{Finsler} norm on $V$ is "almost" a norm in the usual sense; the
difference consists in the fact that it is only \textit{positively }%
homogeneous, instead of absolutely homogeneous. The precise definition is
given below.

\begin{definition}
(\cite{Bao}): A Finsler norm on the vector space $V$ is a function $%
F:V\rightarrow \lbrack 0,\infty )$ with the following properties:

\begin{enumerate}
\item $F$ is smooth on $\mathcal{T}${\textbf{$=$}}$V\backslash \{0\}$ and
continuous at $v=0;$

\item $F$ is positively homogeneous of degree 1, i.e., $F(\lambda v)=\lambda
F(v),$ $\forall \lambda >0;$

\item For every $v\in \mathcal{T},$ the fundamental tensor $g_{v}:V\times
V\rightarrow \mathbb{R},$ 
\begin{equation}
\ g_{v}(u,w):=\dfrac{1}{2}\dfrac{\partial ^{2}F^{2}}{\partial t\partial s}%
(v+tu+ws)\mid _{t=s=0}  \label{g_ definition}
\end{equation}%
is positive definite.
\end{enumerate}
\end{definition}

\bigskip

\textbf{Note: }In the Finsler geometry literature, a function $F$ with the
above properties is called a \textit{Minkwoski norm}. Yet, in order to avoid
confusions with the \textit{Minkowski metric} $\eta $ as defined above, we
will avoid this terminology here and call $F$ instead, a \textsl{Finsler}%
\textit{\ }norm.

\bigskip

\textbf{Particular case. }Euclidean spaces are recovered for $F(v)=\sqrt{%
a_{ij}v^{i}v^{j}},$ where $\left( a_{ij}\right) $ is a constant matrix
(i.e., $g_{v}=a$ does not depend on $v$). In this case, the Finsler norm $F$
is a Euclidean one, since it arises from a scalar product.

\bigskip

Yet, generally, a Finsler norm does not generally arise from a scalar
product. Nevertheless, there exists a similar notion to a scalar product -
namely, the \textit{fundamental }(or \textit{metric})\textit{\ tensor }$%
g_{v} $ - but, in general $g_{v}$ has a nontrivial dependence on the vector $%
v.$ More precisely, the fundamental tensor of the Finsler space $(V,F)$ is
the mapping $g:V\backslash \{0\}\rightarrow T_{2}^{0}(V),$ $v\mapsto g_{v},$
which associates to each vector $v$ the symmetric and positive definite
bilinear form $g_{v}$ defined above. With respect to an arbitrary basis $%
\left\{ e_{i}\right\} _{i=\overline{0,n}}$ of $V,$ the fundamental tensor $%
g_{v}$ has the matrix:%
\begin{equation}
g_{ij}(v)=\dfrac{1}{2}\dfrac{\partial ^{2}F^{2}}{\partial v^{i}\partial v^{j}%
}(v).  \label{metric tensor}
\end{equation}%
that is:%
\begin{equation}
g_{v}(u,w)=g_{ij}(v)u^{i}w^{j}.  \label{g_v bilinear}
\end{equation}%
Hence, for each $v\in V,~g_{v}$ is a scalar product (with "reference vector" 
$v$) on $V.$ Moreover, due to the homogeneity of $F,$ there holds a similar
formula to the one in Euclidean geometry: 
\begin{equation}
F(v)=\sqrt{g_{v}(v,v)}.  \label{F in terms of g_v}
\end{equation}

In the following, we denote the derivatives of $F$ with subscripts:\ $F_{i}:=%
\dfrac{\partial F}{\partial v^{i}},$ $F_{ij}=\dfrac{\partial ^{2}F}{\partial
v^{i}\partial v^{j}}$ etc.

At any $v\in V\backslash \{0\},$ the Hessian of $F:$ 
\begin{equation}
F_{ij}(v)=\dfrac{1}{F}[g_{ij}(v)-F_{i}(v)F_{j}(v)],  \label{angular metric}
\end{equation}%
is positive semidefinite, with radical spanned by $v.$ This fact serves to
prove, (see \cite{Bao}, p. 8-9):

\begin{enumerate}
\item \textit{the fundamental (or Cauchy-Schwarz) inequality:} 
\begin{equation}
dF_{v}(w)\leq F(w),~\ \forall v,w\in V\backslash \{0\};
\label{pos def CS Finsler}
\end{equation}

\item \textit{the triangle inequality: } 
\begin{equation}
F(v+w)\leq F(v)+F(w),~~\forall v,w\in V.
\label{pos def triangle ineq Finsler}
\end{equation}
\end{enumerate}

The above inequalities are strict, i.e., equality only holds when $v$ and $w$
are collinear.

\bigskip

With respect to a given basis, the fundamental inequality takes the form:%
\begin{equation}
F_{i}(v)w^{i}\leq F(w).  \label{pos_def_CS_FInsler_coords}
\end{equation}%
The name of Cauchy-Schwarz inequality for (\ref{pos_def_CS_FInsler_coords})\
is justified by the following. Noticing that%
\begin{equation}
dF_{v}(w)=F_{i}(v)w^{i}=\dfrac{g_{ij}(v)v^{j}w^{i}}{F(v)}=\dfrac{g_{v}(v,w)}{%
F(v)},  \label{F_w}
\end{equation}%
this inequality can be equivalently written as:%
\begin{equation}
g_{v}(v,w)\leq F(v)F(w),  \label{pos def CS Finsler detailed}
\end{equation}%
i.e., the fundamental inequality (\ref{pos def CS Finsler}) is just a
generalization of the usual Cauchy-Schwarz inequality \eqref{non-reversed CS}%
.

\subsection{Lorentz-Finsler structures}

An important feature of Lorentz-Finsler functions is that, typically, they
can only be defined on a conic subset of $V.$

In the following, by\textbf{\ }a \textit{conic domain} of $V,$ we will mean
an open connected subset $\mathcal{Q}$ of $V\backslash \{0\}$ with the 
\textit{conic property}: 
\begin{equation*}
\forall v\in \mathcal{Q},\forall \lambda >0:~\lambda v\in \mathcal{Q}.
\end{equation*}

The definition below is slightly more general than the one by Javaloyes and
Sanchez, \cite[ p. 21]{Javaloyes2019}:

\begin{definition}
Let $\mathcal{T}\subset V\backslash \{0\}$ be a conic domain. We call a
Lorentz-Finsler norm on $\mathcal{T}$ a smooth function $F:\mathcal{T}%
\rightarrow (0,\infty )$ such that:

\begin{enumerate}
\item $F$ is positively homogeneous of degree 1: $F(\lambda v)=\lambda F(v),$
$\forall \lambda >0,$ $\forall v\in \mathcal{T}.$

\item For every $v\in \mathcal{T},$ the fundamental tensor $g_{v}:V\times
V\rightarrow \mathbb{R},$ 
\begin{equation*}
\ g_{v}(u,w):=\dfrac{1}{2}\dfrac{\partial ^{2}F^{2}}{\partial t\partial s}%
(v+tu+ws)\mid _{t=s=0}
\end{equation*}%
has Lorentzian signature $(+,-,-,...,-)$.
\end{enumerate}
\end{definition}

A Lorentz-Finsler norm can always be continuously extended as 0 at $v=0.$

\bigskip

\textbf{Notes:}

\begin{enumerate}
\item The difference between the above introduced notion and the one of 
\textit{Lorentz-Minkowski norm} presented in \cite{Javaloyes2019} is that we
will \textit{not} require $F$ to be extended as 0 on $\partial \mathcal{T};$
while this requirement is important to applications in physical theories, in
our case, it would just uselessly limit the range of allowed examples (see,
e.g., Section \ref{Section_Popoviciu}). Actually, as we will see in the next
section, we will even allow $g_{v}$ to be degenerate at some vectors $v\in 
\mathcal{T}$.

\item Equipping a differentiable manifold $M$ with a smooth family of
Lorentz Finsler norms $p\mapsto F(p)$ which define a Lorentz Finsler
structure $F(p)$ on each tangent space $T_{p}M,$ $p\in M$, and demanding
that $F|_{\partial \mathcal{T}}=0,$ makes the pair $(M,L=F^{2})$ a Finsler
spacetime, \cite{cosmo-Berwald,Javaloyes2019}. Finsler spacetimes gain
attention in the application to gravitational physics \cite%
{Hohmann:2018rpp,Pfeifer:2019wus}, as well as in the mathematical community
as generalizations of Lorentzian manifolds \cite{Bernal:2020bul}.

\item If, in the above definition, one replaces the condition of Lorentzian
signature with positive definiteness, one obtains the notion of (positive
definite) \textit{conic Finsler metric, \cite{Javaloyes2019}}. Thus, a usual
Finsler metric is a conic Finsler metric with $\mathcal{T}=V\backslash
\{0\}. $
\end{enumerate}

For Lorentz-Finsler norms $F$, the matrix $g_{ij}(v)$ is defined by the same
formula (\ref{metric tensor}) (but this time, it has Lorentzian signature)
and the relation $F(v)=\sqrt{g_{v}(v,v)}$ still holds. The Hessian $F_{ij}$
is negative semidefinite with radical spanned by $v,$ i.e.,%
\begin{equation}
F_{ij}(v)w^{i}w^{j}\leq 0,  \label{Hess(F)_ineq}
\end{equation}%
for all $v\in \mathcal{T},$ and $w\in V$, where equality implies that $w$ is
collinear to $v$. Conversely, if $(F_{ij}(v))$\ is negative semidefinite
with 1-dimensional radical, then $g_{v}\ $has $(+,-,-,...,-)$ signature, (see%
\footnote{%
The proof of (\ref{Hess(F)_ineq}) inside the open conic set $\mathcal{T}$ in
the cited paper does not require $F$ to be extendable as $0$ on $\partial T,$
hence the result holds with no modification in our case\textbf{.}} \cite%
{Javaloyes2019}, Proposition 4.8 and, respectively, Lemma 4.7)\textbf{.}

\bigskip

\textbf{Examples of Lorentz-Finsler norms. }Here we just briefly list some
examples $F:\mathcal{T}\rightarrow \mathbb{R}$ (defined on conic subsets $%
\mathcal{T\subset }\mathbb{R}^{n+1}$), to be examined in the following
sections.

\begin{enumerate}
\item The $(n+1)$-dimensional \textit{Minkowski metric}: $F(v)=\sqrt{\eta
_{ij}v^{i}v^{j}}.$

\item $\left( \alpha ,\beta \right) $-\textit{spacetime metrics: }$%
F(v)=\varphi (s)\sqrt{\eta _{ij}v^{i}v^{j}}$, where $s=\dfrac{b_{i}v^{i}}{%
\sqrt{\eta _{ij}v^{i}v^{j}}}$ and $\varphi =\varphi (s)$ is a smooth
function on its domain of definition.

\item \textit{The }$p$\textit{-pseudo-norm}: $F(v)=$ $\left[ \left(
v^{0}\right) ^{p}-\left( v^{1}\right) ^{p}-...-\left( v^{n}\right) ^{p}%
\right] ^{\tfrac{1}{p}}.$

\item The $(n+1)$-dimensional Berwald-Mo\'{o}r metric $%
F(v)=(v^{0}v^{1}...v^{n})^{\tfrac{1}{n+1}}.$

\item Bimetric spaces: $~F(v)=[(\eta _{ij}v^{i}v^{j})(h_{kl}v^{k}v^{l})]^{%
\tfrac{1}{4}}$, where $h_{kl}v^{k}v^{l}$ has Lorentzian signature.
\end{enumerate}

The latter three examples belong to a wider class of Lorentz-Finsler
functions $F$, called $m$\textit{-th root metrics}, expressed as the $m$-th
root of some polynomial of degree $m>2$ in $v^{i}$.

\subsection{The degenerate/non-smooth case\label{Section degenerate case}}

In previous works on the topic, such as \cite%
{Aazami:2014ata,Javaloyes2019,Minguzzi2014,Minguzzi:2013sxa}, the Finslerian
generalizations of the reverse Cauchy-Schwarz inequality and of the reverse
triangle inequality were proven under the hypothesis that the set%
\begin{equation*}
B(1)=F^{-1}([1,\infty ))
\end{equation*}%
is strictly convex (a sufficient condition thereof is that $F$ vanishes on $%
\partial \mathcal{T}$ - which, as we mentioned above, is not assumed here).
These inequalities are strict, i.e., equality happens if and only if $v$ and 
$w$ are collinear. Also, in \cite{Javaloyes2019}, it is proven that, if $%
B(1) $\ is just (non-strictly)\ convex, then the inequalities hold
non-strictly.

\bigskip

In this section, we will present a reformulation - and a slight extension -
of the above results (by relaxing either the nondegeneracy condition on $%
g_{v}$ or the smoothness, even the continuity, assumption on $F$). Also, we
will only require as a hypothesis the convexity of the domain $\mathcal{T}$.

\subsubsection{The non-smooth case}

Let us drop, for the moment, any smoothness (or even continuity) assumption
on $F.$ We obtain the following result.

\begin{proposition}
\label{non-smooth case}For a positively 1-homogeneous function $F:\mathcal{%
T\rightarrow }(0,\infty ),$ $v\mapsto F(v)$ defined on a convex conic domain%
\textbf{\ }$\mathcal{T}\subset V\backslash \{0\},$ the following statements
are equivalent:

(i) $F$ obeys the reverse triangle inequality $F(u+v)\geq F(u)+F(v),$ $%
\forall u,v\in \mathcal{T};$

(ii) $F$ is a concave function;

(iii) the set $B(1)=F^{-1}([1,\infty ))$ is convex;

Moreover, the reverse triangle inequality of $F$ is strict if and only if
the convexity of $B(1)$ is strict.
\end{proposition}

\begin{proof}
\textit{(i) }$\rightarrow $ \textit{(ii): }Assume that\textit{\ }$F$ obeys
the reverse triangle inequality and pick two arbitrary vectors $u,v\in 
\mathcal{T}.$ Then, for any $\alpha \in \lbrack 0,1],$ the convex
combination $(1-\alpha )u+\alpha v$ lies in $\mathcal{T}$ (as $\mathcal{T}$
is assumed to be convex), hence, it makes sense to speak about $F((1-\alpha
)u+\alpha v).$ Using \textit{(i)} and the homogeneity of $F,$ we find: 
\begin{equation}
F((1-\alpha )u+\alpha v)\geq F((1-\alpha )u)+F(\alpha v)=\left( 1-\alpha
\right) F(u)+\alpha F(v),  \label{concave_F}
\end{equation}%
i.e., $F$ is concave.

\textit{(ii) }$\rightarrow $\textit{(i):} If $F$ is concave, then, for any $%
u,v\in \mathcal{T}:$ $F(\dfrac{u+v}{2})\geq \dfrac{1}{2}F(u)+\dfrac{1}{2}%
F(v).$ Using the homogeneity of $F,$ this yields the reverse triangle
inequality (\ref{reverse triangle}).

\textit{(i)}$\rightarrow $\textit{(iii): }Assuming that the reverse triangle
inequality holds, pick two arbitrary vectors $v,w\in B(1)$ (i.e., $%
F(v),F(w)\geq 1$) and an arbitrary $\alpha \in \lbrack 0,1].$ Then,%
\begin{equation*}
F((1-\alpha )v+\alpha w)\geq F((1-\alpha )v)+F(\alpha w)=(1-\alpha
)F(v)+\alpha F(w)\geq 1,
\end{equation*}%
which means that $(1-\alpha )v+\alpha w\in B(1).$ Consequently, $B(1)$ is
convex. Also, if the triangle inequality is strict, then, for all
non-collinear $v,w$, the first inequality above is strict, which eventually
leads to $F((1-\alpha )v+\alpha w)>1,$ i.e., the convexity of $B(1)$ is
strict.

\textit{(iii)}$\rightarrow $\textit{(i): }The idea of the proof is similar
to the one in the positive semi-definite case (see e.g., \cite{Javaloyes2014}%
). Assume $B(1)$ is convex and pick two arbitrary vectors $v,w\in \mathcal{T}%
.$ By the 1-homogeneity of $F,$ it follows that the vectors $v^{\prime }:=%
\dfrac{v}{F(v)},$ $w^{\prime }=\dfrac{w}{F(w)}$ obey $F(v^{\prime
})=F(w^{\prime })=1$, i.e., $v^{\prime },w^{\prime }\in \partial B(1)$. Set $%
\alpha :=\dfrac{F(w)}{F(v)+F(w)}\in (0,1)$ and build the convex combination:%
\begin{equation*}
u:=(1-\alpha )v^{\prime }+\alpha w^{\prime }=\dfrac{v+w}{F(v)+F(w)}\in 
\mathcal{T}.
\end{equation*}

By the homogeneity assumption on $F,$ we have $u\in B(1),$ i.e., $F(u)\geq
1. $ But, using again the homogeneity of $F,$ this is:\ $F(v+w)\geq
F(v)+F(w),$ q.e.d. If $B(1)$ is strictly convex, then, whenever $u,v$ are
non-collinear, $u$ belongs to the interior of $B(1),$ i.e., $F(u)>1.$ This
leads to the strictness of the reverse triangle inequality.
\end{proof}

The above result extends a result in \cite{Aazami:2014ata}, by removing the
restrictions on the continuity of $F$ or on the boundary $\partial \mathcal{T%
}.$

\bigskip

Similarly,\textbf{\ }for a 1-homogeneous function $F:\mathcal{T}\rightarrow 
\mathbb{R},$ there holds (see also \cite{Javaloyes2014}) the equivalence:

\begin{center}
\textit{(non-reversed) triangle inequality} $\Leftrightarrow F$ \textit{-
convex} $\Leftrightarrow F^{-1}([0,1])$ \textit{- convex.}
\end{center}

\bigskip

\textbf{Remark. }For a general positively 1-homogeneous function $F:\mathcal{%
T\rightarrow }\mathbb{R}$ (where $\mathcal{T\in }V\backslash \{0\}$ is
conic), convexity of $B(1)=F^{-1}\left( [1,\infty )\right) $ is a stronger
requirement than convexity of its domain $\mathcal{T}$. Indeed, if $B(1)$ is
convex, then, the homogeneity of $F$ implies that all the sets $%
aB(1)=F^{-1}\left( [a,\infty )\right) ,$ $a>0$ are convex; this immediately
implies the convexity of $\mathcal{T=}F^{-1}\left( (0,\infty )\right) $).
Yet, the converse statement is generally not true; as we have seen above, it
actually depends on the concavity of the function $F$.

\subsubsection{Degenerate Lorentz-Finsler norms}

Now, let us assume again that the function $F$ is smooth. We will first
prove a lemma, which extends (\ref{Hess(F)_ineq}) to degenerate Finsler
structures.

\begin{lemma}
\label{signature_Hess(F)} Consider a smooth, 1-homogeneous function $F:%
\mathcal{T\rightarrow }\mathbb{R}$ defined on a conic domain $\mathcal{T}%
\subset V\backslash \{0\},$ an arbitrary $v\in \mathcal{T}$ and denote, with
respect to an arbitrary basis: 
\begin{equation*}
g_{ij}(v)=\dfrac{1}{2}\dfrac{\partial ^{2}F^{2}}{\partial v^{i}\partial v^{j}%
}(v)
\end{equation*}%
Then:

\begin{enumerate}
\item[\textit{(i)}] The matrix $\left( g_{ij}(v)\right) $ has only one
positive eigenvalue if and only if the Hessian $(F_{ij}(v))$ is negative
semidefinite.

\item[\textit{(ii)}] The matrix $\left( g_{ij}(v)\right) $ is positive
semidefinite if and only if the Hessian $(F_{ij}(v))$ is positive
semidefinite.
\end{enumerate}
\end{lemma}

\begin{proof}

\begin{enumerate}
\item[\textit{(i)}] $\rightarrow :$ Assume $g_{ij}(v)$ has only one positive
eigenvalue. Using $g_{ij}(v)=F(v)F_{ij}(v)+F_{i}(v)F_{j}(v),$ we find for
any $u\in V:$%
\begin{equation}
F(v)F_{ij}(v)u^{i}u^{j}=g_{ij}(v)u^{i}u^{j}-(F_{i}(v)u^{i})^{2}.
\label{ineq_F_g}
\end{equation}%
As the signature of $g_{v}$ does not depend on the choice of the basis $%
\left\{ e_{i}\right\} _{i=\overline{0,n}}$, we can freely choose this basis.
For instance, we can choose an orthogonal basis for $g_{v},$\ with $e_{0}=v.$
Since $g_{v}(e_{0},e_{0})=g_{ij}(v)v^{i}v^{j}=F^{2}(v)>0,$ it follows from
the hypothesis that all the other diagonal entries $g_{ii}(v)$ are
nonpositive. Setting $u=e_{\alpha }$ for $\alpha \not=0,$ the orthogonality
condition is written, taking into account (\ref{F_w}), as $F_{i}(v)u^{i}=0;$
therefore, 
\begin{equation}
F_{ij}(v)u^{i}u^{j}=\dfrac{1}{F(v)}g_{ij}(v)u^{i}u^{j},  \label{ineq_F_g_1}
\end{equation}%
which entails $F_{ij}(v)u^{i}u^{j}\leq 0.$ But, on the other hand, we have: $%
F_{ij}(v)e_{0}^{i}e_{0}^{j}=F_{ij}(v)v^{i}v^{j}=0;$ that is, evaluating the
bilinear form $F_{ij}(v)$ on any basis vector $e_{i}$, we get nonpositive
values, i.e., $F_{ij}(v)$ is negative semidefinite for any $v\in \mathcal{T}$%
.

$\leftarrow :$ Conversely, assume $F_{ij}(v)$ is negative semidefinite.
Using the same $g_{v}$-orthogonal basis as above, we find from (\ref%
{ineq_F_g_1}) that, for $u=e_{\alpha }$, $\alpha =1,...,n,$ there holds $%
0\geq g_{ij}(v)u^{i}u^{j},$ i.e., $g_{ij}(v)$ has $n$ nonpositive
eigenvalues. But, on the other hand, $g_{v}(e_{0},e_{0})=F^{2}(v)>0,$ i.e.,
the eigenvector $e_{0}=v$ corresponds to a (unique) positive eigenvalue for $%
g.$

\item[\textit{(ii)}] is proven similarly, taking into account that, this
time, $F(v)F_{ij}(v)u^{i}u^{j}=g_{ij}(v)u^{i}u^{j}\geq 0.$
\end{enumerate}
\end{proof}

\bigskip

From the above Lemma and Proposition \ref{non-smooth case}, we immediately
find:

\begin{theorem}
\label{full_equiv}For a smooth, positively 1-homogeneous function $F:%
\mathcal{T}\rightarrow (0,\infty )$ defined on a convex conic domain $%
\mathcal{T}\subset V\backslash \{0\},$ the following statements are
equivalent:

\begin{enumerate}
\item[\textit{(i)}] $F$ - concave $\Leftrightarrow $ $g_{v}$ has exactly one
positive eigenvalue $\Leftrightarrow $ $F$ obeys the reverse triangle
inequality $\Leftrightarrow $ the set $B(1)=F^{-1}([1,\infty ))$ is convex;

\item[\textit{(ii)}] $F$ - convex $\Leftrightarrow $ $g_{v}$ is positive
semidefinite $\Leftrightarrow $ $F$ obeys the triangle inequality $%
\Leftrightarrow $ the set $F^{-1}([0,1])$ is convex (where, in the latter,
we have defined $F(0):=0$).
\end{enumerate}
\end{theorem}

In the above Theorem, the triangle inequality respectively, its reversed
counterpart, are generally non-strict.

In \cite{Javaloyes2019}, convexity of $F^{-1}([1,\infty ))$ was proven to be
(also) equivalent to $g_{v}$ being negative semidefinite on the set $%
F^{-1}(\{1\}).$

\bigskip

Finally, we can state (albeit in a somewhat redundant way):

\begin{theorem}
\label{prop_degenerate Finsler case}\textbf{(The degenerate-Lorentzian case):%
} Let $\mathcal{T}\subset V\backslash \{0\}$ be a convex conic domain and $F:%
\mathcal{T\rightarrow }(0,\infty )$ a smooth, positively 1-homogeneous
function. If the Hessian $g_{v}$ of $F^{2}$ has only one positive \textbf{\ }%
eigenvalue for all $v\in \mathcal{T}$, then, for any\textbf{\ }$v,w\in 
\mathcal{T},$ there hold:

\begin{enumerate}
\item[\textit{(i)}] the fundamental (or reverse Cauchy-Schwarz) inequality:%
\begin{equation}
dF_{v}(w)\geq F(w).  \label{reverse_CS_coord_free}
\end{equation}

\item[\textit{(ii)}] The reverse triangle inequality:%
\begin{equation}
F(v+w)\geq F(v)+F(w);  \label{reverse triangle}
\end{equation}%
If $g_{v}$ is everywhere nondegenerate (i.e., Lorentzian), then both the
above inequalities are strict.
\end{enumerate}
\end{theorem}

\begin{proof}

\begin{enumerate}
\item[\textit{(i)}] The technique follows roughly the same steps as in the
positive definite case (see, e.g., \cite{Bao}, p. 8-9). Consider two
arbitrary vectors $u,v\in \mathcal{T}.$ Since $\mathcal{T}$ is convex, it
follows that $\dfrac{u+v}{2}\in \mathcal{T};$ but as it is also conic, we
find $u+v\in \mathcal{T},$ which means that it makes sense to speak about $%
F(u+v)$. Now, perform a Taylor expansion around $v$, with the remainder in
Lagrange form:%
\begin{equation}
F(u+v)=F(v)+F_{i}(v)u^{i}+\dfrac{1}{2}F_{ij}(v+\varepsilon u)u^{i}u^{j}.
\label{Taylor1}
\end{equation}%
From the above Lemma, we obtain that $F_{ij}$ is negative semidefinite, that
is, $F_{ij}(v+\varepsilon u)u^{i}u^{j}\leq 0$ and therefore,%
\begin{equation}
F(u+v)\leq F(v)+F_{i}(v)u^{i}.  \label{aux}
\end{equation}

Then, denoting $w:=u+v,$ the above becomes $F(w)\leq
F(v)+F_{i}(v)(w^{i}-v^{i}),$ which, using the 1-homogeneity of $F,$ leads
to: $F(w)\leq F_{i}(v)w^{i},$ which is the coordinate form of \textit{(i).}

\item[\textit{(ii)}] The reverse triangle inequality follows immediately
from Theorem \ref{full_equiv}.

\textit{Strictness:}\ Assume $g_{v}$ is nondegenerate. Then, the equality $%
F_{ij}(v+\varepsilon u)u^{i}u^{j}=0$ can only happen when $v$ and $u$ are
collinear; this leads to the strictness (\ref{aux}) and consequently, of (%
\ref{reverse_CS_coord_free}). Further, assume $v,w\in \mathcal{T}$ are
non-collinear and set $\xi :=v+w$. Then, 
\begin{equation*}
F(\xi )=F_{i}(\xi )\xi ^{i}=F_{i}(\xi )(v^{i}+w^{i})=F_{i}(\xi
)v^{i}+F_{i}(\xi )w^{i}.
\end{equation*}%
Applying the reverse Cauchy-Schwarz inequality twice in the right hand side,
we get: $F(v+w)=F(\xi )>F(v)+F(w),$ i.e., the reverse triangle inequality is
strict.
\end{enumerate}
\end{proof}

\bigskip

Using (\ref{F_w}), the fundamental inequality can be equivalently written as:%
\begin{equation}
F_{i}(v)w^{i}\geq F(w),  \label{CS_coords}
\end{equation}%
or as:%
\begin{equation}
g_{v}(v,w)\geq F(v)F(w).  \label{reverse_CS_gv}
\end{equation}

\bigskip

\textbf{Example. }To convince ourselves that the reverse Cauchy-Schwarz
inequality becomes non-strict if $F\ $is degenerate, consider 
\begin{equation*}
F:\mathcal{T}\rightarrow \mathbb{R},~~F(v)=\sqrt{\left( v^{0}\right)
^{2}-\left( v^{1}\right) ^{2}-....-\left( v^{k}\right) ^{2}};
\end{equation*}%
here, $n>3,~k\leq n-2$ and the cone $\mathcal{T}\subset \mathbb{R}^{n+1}$ is
the Cartesian product $\mathcal{T}=\mathcal{T}_{k}\times \mathbb{R}^{n-k},$
where $\mathcal{T}_{k}=\{u\in \mathbb{R}^{k+1}~|~\left( u^{0}\right)
^{2}-\left( u^{1}\right) ^{2}-...-\left( u^{k}\right) ^{2}>0,~u^{0}>0\}.$
Since $\mathcal{T}_{k}$ is a convex cone in $\mathbb{R}^{k+1}$, it follows
that $\mathcal{T}$ is also convex. The corresponding metric tensor is $%
g_{v}=diag(1,-1,...,-1,0,...,0),$ where the number of $-1$ entries is $k.$
Picking $v=(1,0,0,....0,1,0)$ and $w=(1,0,0,....,0,1),$ we get: $%
g_{v}(v,w)=1,$ $F(v)=1$ and $F(w)=1,$ which means that $g_{v}(v,w)=F(v)F(w),$
while, obviously, $v$ and $w$ are not collinear.

In the maximally degenerate case, when $g_{v}$ has everywhere signature $%
(+,0,0,...,0),$ $(F_{ij}(v))$ is the zero matrix, hence, (\ref%
{reverse_CS_coord_free}) and (\ref{reverse triangle}) become equalities for 
\textit{all }$v,w\in \mathcal{T}.$

\bigskip

\textbf{Remark. }Generally, if $F$ is not smooth, it makes no sense to ask
about the (reverse)\ fundamental inequality, as it involves $dF_{v}$. Yet,
if for some fixed $v\in \mathcal{T},$ the function $F$ is concave and $%
\mathcal{C}^{2}$-smooth on a convex neighborhood $U\subset \mathcal{T}$ of $%
v $ then, any $w\in U,$ then (\ref{Taylor1})-(\ref{aux}) still hold for $v$
and $u=w-v;$ that is, the reverse fundamental inequality will hold on the
neighborhood $U.$

\bigskip

Similarly, it holds:

\begin{proposition}
\label{positive semidefinite case}\textbf{(The positive semidefinite case):}
Let $\mathcal{T}\subset V\backslash \{0\}$ be a convex conic domain and $F:%
\mathcal{T\rightarrow }(0,\infty ),$ a smooth positively 1-homogeneous
function such that the Hessian $g_{v}$ of $F^{2}$ is positive semidefinite
for all $v\in \mathcal{T}$. Then, the Cauchy-Schwarz inequality (\ref{pos
def CS Finsler})\ and the triangle inequality (\ref{pos def triangle ineq
Finsler}) still hold, but they are generally non-strict.
\end{proposition}

The proof is identical to the one of Proposition \ref{prop_degenerate
Finsler case}, with the only difference that, in (\ref{Taylor1}), the matrix 
$F_{ij}(v+\varepsilon u)$ is positive semidefinite, which leads to the
opposite inequality: $F(u+v)\geq F(v)+F_{i}(v)u^{i},$ hence, to the usual
(non-reversed)\ Cauchy-Schwarz and triangle inequalities.

\subsection{Refinements of the Finslerian reverse triangle inequality\label%
{sec: refinements triangle}}

Here are two refinements of the reverse triangle inequality, holding for
(possibly degenerate, or non-smooth) Lorentz-Finsler functions.

\begin{theorem}
If a positively 1-homogeneous\textbf{\ }function $F:\mathcal{T}\rightarrow
(0,\infty ),$ defined on a convex conic domain $\mathcal{T}\subset
V\backslash \{0\},$ obeys the reverse triangle inequality, then, for all $%
v,w\in \mathcal{T}$ and for any $0<a\leq b,$ we have:%
\begin{equation}
a\left[ F(v+w)-F(v)-F(w)\right] \leq F(av+bw)-aF(v)-bF(w)\leq
b[F(v+w)-F(v)-F(w)].  \label{reverse triangle_1}
\end{equation}%
If the reverse triangle inequality of $F$ is strict, then the above
inequalities are also strict.
\end{theorem}

\begin{proof}
The first inequality is equivalent (after canceling out the $-aF(v)$ terms
and grouping the $F(w)$ ones into the left hand side) to: 
\begin{equation*}
aF(v+w)+(b-a)F(w)\leq F(av+bw).
\end{equation*}

But, since $F$ is positively homogeneous and $b-a\geq 0,$ we get: $%
aF(v+w)=F(av+aw)$ and $(b-a)F(w)=F(bw-aw)$. Then, the reverse triangle
inequality yields: 
\begin{equation*}
aF(v+w)+(b-a)F(w)=F(av+aw)+F(bw-aw)\leq F(av+bw)
\end{equation*}%
as required.

The second inequality is proven in a completely similar way to be equivalent
to: $F(av+bw)+F(bv-av)\leq F(bv+bw),$ which, again, holds by virtue of the
reverse triangle inequality.
\end{proof}

\begin{proposition}
If a continuous, positively 1-homogeneous function $F:\mathcal{T}\rightarrow 
\mathbb{R}$ defined on a convex conic domain $\mathcal{T}\subset V\backslash
\{0\},$ obeys the reverse triangle inequality, then: 
\begin{equation}
F(v)+F(w)\leq 2\int_{0}^{1}F(tv+(1-t)w)dt\leq F\left( v+w\right) ,
\label{reverse triangle_3}
\end{equation}%
for all $v,w\in \mathcal{T\subset }$ $V\backslash \{0\}$.
\end{proposition}

\begin{proof}
We use the same idea as in \cite{Min-Pal}. Using the reverse triangle
inequality and homogeneity, we find: 
\begin{equation*}
F(tv+(1-t)w))\geq F(tv)+F((1-t)w))=tF(v)+(1-t)F(w),
\end{equation*}%
for every $v,w\in \mathcal{T}$, $t\in \lbrack 0,1]$. Integrating with
respect to $t,$ from $0$ to $1$, we obtain: 
\begin{equation*}
\frac{F(v)+F(w)}{2}\leq \int_{0}^{1}F(tv+(1-t)w)dt,
\end{equation*}%
i.e., the first inequality (\ref{reverse triangle_3}). Similarly, using the
reverse triangle inequality, we have $F(v+w)=F{\large (}v+(1-t)w+(1-t)v+tw%
{\large )}\geq F(tv+(1-t)w)+F((1-t)v+tw)$. Integrating from $0$ to $1$, we
deduce: $F(v+w)\geq
\int_{0}^{1}F(tv+(1-t)w)dt+\int_{0}^{1}F((1-t)v+tw)dt=2%
\int_{0}^{1}F(tv+(1-t)w)dt$, which is just the second inequality (\ref%
{reverse triangle_3}).
\end{proof}

The two above results trivially hold when one of the vectors $v,w$ is zero.

\section{Lorentz-Finsler norms and their inequalities}

\label{sec:ex}

The set of Lorentz-Finsler norms is rich in interesting examples whose
reverse Cauchy-Schwarz or reverse\textbf{\ }triangle inequality yield
immediately famous inequalities from the literature and open a pathway to
reveal further interesting inequalities.

We already pointed out in Section \ref{sec:classineq} that, for the simplest
example of Lorentz-Finsler structure on $\mathbb{R}^{n+1}$, the Minkowski
metric $F(v)=\sqrt{\eta _{ij}v^{i}v^{j}}$, for which $\mathcal{T=}\left\{
v\in V~|~\eta _{ij}v^{i}v^{j}>0,v^{0}>0\right\} $, its reverse
Cauchy-Schwarz inequality led directly to Acz\'{e}l's inequality (\ref{1_1}).

In the following, we explore some nontrivial Finslerian cases.

\subsection{Popoviciu's inequality\label{Section_Popoviciu}}

\begin{proposition}
Let $\mathcal{T}\subset \mathbb{R}^{n+1}$ be the conic domain: 
\begin{equation*}
\mathcal{T}:=\{v\in \mathbb{R}^{n+1}~|~v^{0},v^{1},...,v^{n}>0,\left(
v^{0}\right) ^{p}-\left( v^{1}\right) ^{p}-...-\left( v^{n}\right) ^{p}>0\}.
\end{equation*}%
Moreover let $F:\mathcal{T}\rightarrow \mathbb{R}^{+}$ be the
Lorentz-Finsler structure defined by 
\begin{equation}
F(v)=H(v)^{\frac{1}{p}},\quad H(v)=\left( v^{0}\right) ^{p}-\left(
v^{1}\right) ^{p}-...-\left( v^{n}\right) ^{p}\,,  \label{p-pseudo-norm}
\end{equation}%
where $p>1$. Then:

\begin{enumerate}
\item[\textit{(i)}] the fundamental inequality $F_{i}(v)w^{i}\geq F(w)$ is
Popoviciu's inequality: 
\begin{equation}
\eta _{ij}a^{i}b^{j}\geq \left[ (a^{0})^{q}-\left( a^{1}\right)
^{q}-...-\left( a^{n}\right) ^{q}\right] ^{\frac{1}{q}}\left[ \left(
b^{0}\right) ^{p}-\left( b^{1}\right) ^{p}-...-\left( b^{n}\right) ^{p}%
\right] ^{\frac{1}{p}},\ \forall a,b\in \mathcal{T}\,,
\label{Popoviciu inequality}
\end{equation}%
where $\dfrac{1}{p}+\dfrac{1}{q}=1$;

\item[\textit{(ii)}] the reverse triangle inequality of $F$ is Bellman's
inequality: 
\begin{equation}
\left( v_{0}^{p}-v_{1}^{p}-...-v_{n}^{p}\right) ^{1/p}+\left(
w_{0}^{p}-w_{1}^{p}-...-w_{n}^{p}\right) ^{1/p}\leq \lbrack
(v_{0}+w_{0})^{p}-(v_{1}+w_{1})^{p}-...-(v_{n}+w_{n})^{p}]^{1/p}.
\label{Bellmann}
\end{equation}
\end{enumerate}
\end{proposition}

\begin{proof}

\begin{enumerate}
\item[\textit{(i)}] To see that the fundamental inequality holds, we realise
that the Hessian of $H$ is\textbf{\ }%
\begin{equation*}
H_{ij}(v)=p(p-1)diag\left(
(v^{0})^{p-2},-(v^{1})^{p-2},....,-(v^{n})^{p-2}\right) .
\end{equation*}%
and thus has Lorentzian signature on $\mathcal{T}$. Moreover, $\mathcal{T}$
is convex (but not strictly convex), as it can be identified with the
epigraph of the convex function $\tilde{H}%
(v^{1},...,v^{n})=(v^{1})^{p}+....+\left( v^{n}\right) ^{p}$, defined for
all $v^{\alpha }>0$.

By Proposition~\ref{prop:signHg} (see Appendix), we obtain that $g_{ij}(v)$
is Lorentzian for all $v\in \mathcal{T}$ and hence, the fundamental
inequality inequality (\ref{CS_coords}) holds.

By a straightforward calculation, we get: 
\begin{equation*}
F_{i}(v)=F(v)^{1-p}\eta _{ij}(v^{j})^{p-1};
\end{equation*}%
therefore, the fundamental inequality $F_{i}(v)w^{i}\geq F(w)$ becomes: 
\begin{equation*}
\eta _{ij}(v^{j})^{p-1}w^{i}\geq F(v)^{p-1}F(w)=H(v)^{\frac{p-1}{p}}H(w)^{%
\frac{1}{p}}.
\end{equation*}%
Evaluating the above equation with the following notation: 
\begin{equation*}
q:=\dfrac{p}{p-1},\quad a^{i}:=\left( v^{i}\right) ^{p-1},\quad \
b^{j}:=w^{j},
\end{equation*}%
(in particular, $\dfrac{1}{p}+\dfrac{1}{q}=1$), yields Popoviciu's
inequality.

\item[\textit{(ii)}] is obvious.
\end{enumerate}
\end{proof}

\begin{remark}
Similarly, H\"{o}lder's inequality 
\begin{equation*}
\delta _{ij}a^{i}b^{j}\leq \left[ \left( a^{0}\right) ^{q}+...+\left(
a^{n}\right) ^{q}\right] ^{\frac{1}{q}}\left[ \left( b^{0}\right)
^{p}+...+\left( b^{n}\right) ^{p}\right] ^{\frac{1}{p}},\ \ \ \ \ \forall
a^{i},b^{i}>0,i=\overline{0,n},
\end{equation*}%
where $\delta _{ij}$ is the Kronecker symbol, can be treated as fundamental
inequality of the Finsler norm $F(v)=[\left( v^{0}\right) ^{p}+...+\left(
v^{n}\right) ^{p}]^{\tfrac{1}{p}}$ - which is positive definite for $%
v^{i}>0, $ $i=\overline{0,n}$ and Minkowski's inequality%
\begin{equation*}
\left[ \left( a^{0}+b^{0}\right) ^{p}+...+\left( a^{n}+b^{n}\right) ^{p}%
\right] ^{\frac{1}{p}}\leq \left[ \left( a^{0}\right) ^{p}+...+\left(
a^{n}\right) ^{p}\right] ^{\frac{1}{p}}+\left[ \left( b^{0}\right)
^{p}+...+\left( b^{n}\right) ^{p}\right] ^{\frac{1}{p}},
\end{equation*}%
$\forall a^{i},b^{i}>0,i=\overline{0,n},p>1$ is just the corresponding
triangle inequality.
\end{remark}

\subsection{The arithmetic-geometric mean inequality\label{Section_BM}}

\begin{proposition}
Let $\mathcal{T}\subset \mathbb{R}^{n+1}$ be the convex conic domain 
\begin{equation*}
\mathcal{T}:=\left\{ v\in \mathbb{R}^{n+1}~|~v^{0},v^{1},...,v^{n}>0\right\}
\subset \mathbb{R}^{n+1}\,.
\end{equation*}%
Moreover let $F:\mathcal{T}\rightarrow \mathbb{R}^{+}$ be the Berwald-Mo\'{o}%
r Finsler structure defined by 
\begin{equation*}
F(v)=(v^{0}v^{1}...v^{n})^{\tfrac{1}{n+1}}\,.
\end{equation*}%
Then, the fundamental inequality $F_{i}(v)w^{i}\geq F(w)$ is the
aritmetic-geometric mean inequality: 
\begin{equation}
\dfrac{a_{0}+....+a_{n}}{n+1}\geq \left( a_{0}a_{1}...a_{n}\right) ^{\tfrac{1%
}{n+1}},~\ \ \forall a_{i}\in \mathbb{R}_{+}^{\ast }\,.
\label{arithmetic-geometric mean}
\end{equation}
\end{proposition}

\begin{proof}
The $n$-dimensional Berwald-Mo\'{o}r metric is known, \cite{Asanov1980}, to
be of Lorentzian signature. Yet, for the sake of completeness, we sketch a
proof of this fact below. To this aim, we will use Proposition \ref%
{prop:signHg}\textbf{.}

The Hessian of the $(n+1)$-th power $H(v):=v^{0}v^{1}...v^{n+1}$ of $F$ is: 
\begin{equation}
H_{ij}(v)=\left\{ 
\begin{array}{c}
0,~\ \ if~\ i=j \\ 
\dfrac{H(v)}{v^{i}v^{j}},\,\ \ \ \ if~~\ i\not=j.%
\end{array}%
\right.  \label{BM_H}
\end{equation}%
On $\mathcal{T},$ the matrix $(H_{ij}(v))$ has Lorentzian signature. To see
this, fix an arbitrary\textbf{\ }$v\in \mathcal{T}$ and introduce the
vectors $e_{0}:=v$ and$\ \ \{e_{\alpha }\},$ $\alpha =\overline{1,n}$ as
follows: 
\begin{equation*}
e_{\alpha }^{i}=A_{\alpha }^{i}v^{i},~\ i=\overline{0,n}\text{ }
\end{equation*}%
(where no summation is understood over $i$), such that: 
\begin{equation*}
\sum_{i=0}^{n}A_{\alpha }^{i}=A_{\alpha }^{0}+\sum_{\beta =1}^{n}A_{\alpha
}^{\beta }=0,~\ \ \ \ \ \det (A_{\alpha }^{\beta })_{\alpha ,\beta =%
\overline{1,n}}\not=0.
\end{equation*}%
The vectors $\left\{ e_{0},e_{\alpha }\right\} $ are linearly independent,
as the matrix with the columns $e_{0},e_{\alpha }$ has the determinant $%
H(v)\det (A_{\alpha }^{\beta })\not=0$. Moreover, $e_{\alpha }$ span the $%
(H_{ij}(v))$-orthogonal complement of $e_{0}=v$, since, using (\ref{BM_H}),
we find: 
\begin{equation*}
H_{ij}(v)v^{i}e_{\alpha }^{j}=0,~\ \ \forall \alpha =1,...,n.
\end{equation*}%
Then, on one hand, we have:%
\begin{equation*}
H_{ij}(v)v^{i}v^{j}=n(n-1)H(v)>0
\end{equation*}%
and, on the other hand, $H_{ij}(v)$ is negative definite on $Span\{e_{\alpha
}\}$, since: 
\begin{equation*}
H_{ij}(v)e_{\alpha }^{i}e_{\alpha }^{j}=H(v)\sum_{i\neq j}\frac{v^{i}}{v^{i}}%
\frac{v^{j}}{v^{j}}A_{\alpha }^{i}A_{\alpha }^{j}=H(v)\left( \left(
\sum_{i=0}^{n}A_{\alpha }^{i}\right) ^{2}-\sum_{i=0}^{n}(A_{\alpha
}^{i})^{2}\right) =0-H(v)\sum_{i=0}^{n}(A_{\alpha }^{i})^{2}\,.
\end{equation*}%
Consequently, $H_{ij}(v)$ has Lorentzian signature. Then, by Proposition \ref%
{prop:signHg}, also $g_{ij}(v)$ has Lorentzian signature on $\mathcal{T}$
and the fundamental inequality $F_{i}(v)w^{i}\geq F(w)$ holds $\forall
v,w\in \mathcal{T}$. We easily find: 
\begin{equation*}
F_{i}(v)=\frac{1}{n+1}H^{\frac{1}{n+1}-1}\frac{H(v)}{v^{i}}=\frac{F(v)}{n+1}%
\frac{1}{v^{i}}
\end{equation*}%
and thus 
\begin{equation*}
\frac{F(v)}{n+1}\sum_{i=0}^{n}\frac{w^{i}}{v^{i}}\geq F(w)
\end{equation*}%
or equivalently 
\begin{equation*}
\frac{1}{n+1}\sum_{i=0}^{n}\frac{w^{i}}{v^{i}}\geq \frac{F(w)}{F(v)}=\left( 
\dfrac{w^{0}}{v^{0}}\dfrac{w^{1}}{v^{1}}...\dfrac{w^{n}}{v^{n}}\right) ^{%
\tfrac{1}{n+1}}\,.
\end{equation*}%
Setting $a_{i}:=\dfrac{w^{i}}{v^{i}},\ i=\overline{0,n}$, $a_{i}$ take all
possible values in $\mathbb{R}_{+}^{\ast }$ and the fundamental inequality
becomes the aritmetic-geometric mean inequality 
\eqref{arithmetic-geometric
mean}.
\end{proof}

\subsection{Weighted arithmetic-geometric mean inequality\label%
{Section_weighted}}

\begin{proposition}
Let $\mathcal{T}\subset \mathbb{R}^{n+1}$ be the convex conic domain 
\begin{equation*}
\mathcal{T}:=\left\{ v\in \mathbb{R}^{n+1}~|~v^{0},v^{1},...,v^{n}>0\right\}
\subset \mathbb{R}^{n+1}\,.
\end{equation*}%
Then, the function $F:\mathcal{T}\rightarrow \mathbb{R}^{+},$ defined by%
\begin{equation*}
F(v)=\left( v^{0}\right) ^{a_{0}}\left( v^{1}\right) ^{a_{1}}...\left(
v^{n}\right) ^{a_{n}},\ \sum_{i=0}^{n}a_{i}=1\quad a_{i}\geq 0\,.
\end{equation*}%
is a Lorentz-Finsler norm whose fundamental inequality is the weighted
arithmetic-geometric mean inequality 
\begin{equation*}
\sum_{i=0}^{n}a_{i}v^{i}\geq
(v^{0})^{a_{0}}(v^{1})^{a_{1}}...(v^{n})^{a_{n}},~\ \ \ \ \ \ \ \ \ \ \
\forall v^{i}\in \mathbb{R}_{+}^{\ast }\,.
\end{equation*}
\end{proposition}

\begin{proof}
Fix $u\in \mathcal{T}.$ The components of the fundamental tensor are (no sum
convention is employed in the following expression) 
\begin{equation*}
g_{ij}(u)=\frac{1}{2}\frac{\partial ^{2}F^{2}}{\partial u^{i}\partial u^{j}}%
=F(u)^{2}\left( \frac{2a_{i}a_{j}}{u^{i}u^{j}}-\frac{a_{i}\delta _{ij}}{%
(u^{i})^{2}}\right) \,.
\end{equation*}%
Its signature can be determined similarly to the previous case. Introducing
the vectors $e_{0}=u$ and $e_{\alpha },\alpha =\overline{1,n}$ with
components 
\begin{equation*}
e_{\alpha }^{i}=B_{\alpha }^{i}u^{i}
\end{equation*}%
such that 
\begin{equation}
\sum_{i=0}^{n}a_{i}B_{\alpha }^{i}=a_{0}B_{\alpha }^{0}+\sum_{\beta
=1}^{n}a_{\beta }B_{\alpha }^{\beta }=0,  \label{B}
\end{equation}%
and $\det (B_{\alpha }^{\beta })\not=0,$ the vectors $\left\{
e_{0},e_{\alpha }\right\} $ are linearly independent and 
\begin{equation*}
g_{ij}(u)e_{0}^{i}e_{\alpha }^{j}=0\,.
\end{equation*}%
Thus, $\{e_{\alpha }\}_{\alpha =\overline{1,n}}$ span the orthogonal
complement of $e_{0}=u$. Moreover, $g_{ij}(u)$ is negative definite on $%
Span\{e_{\alpha }\}$, since, by (\ref{B}), we have: 
\begin{equation*}
g_{ij}(u)e_{\alpha }^{i}e_{\alpha }^{j}=F(u)^{2}\left( 2\left(
\sum_{i=0}^{n}a_{i}B_{\alpha }^{i}\right) ^{2}-\sum_{i=0}^{n}a_{i}(B_{\alpha
}^{i})^{2}\right) =-F(u)^{2}\sum_{i=0}^{n}a_{i}(B_{\alpha }^{i})^{2}<0;
\end{equation*}%
taking into account that $g_{ij}(u)e_{0}^{i}e_{0}^{j}=F(u)^{2}>0,$ we obtain
that $g_{ij}(u)\ $is Lorentzian. Calculating 
\begin{equation*}
F_{i}(u)=a_{i}\frac{F(u)}{u^{i}},
\end{equation*}%
the fundamental inequality \eqref{CS_coords} becomes 
\begin{equation*}
F(u)\sum_{i=0}^{n}\frac{a_{i}w^{i}}{u^{i}}\geq F(w)\,.
\end{equation*}%
Introducing $v_{i}=\frac{w^{i}}{u^{i}}\in \mathbb{R}_{+}^{\ast },$ it can be
rewritten as: $\sum\limits_{i=0}^{n}a_{i}v^{i}\geq
(v^{0})^{a_{0}}(v^{1})^{a_{1}}...(v^{n})^{a_{n}}\,.$The generalization (\ref%
{weighted_a}) then follows immediately.
\end{proof}

\subsection{\label{Section_bimetric}Bimetric structures on $\mathbb{R}^{n+1}$%
}

Finsler functions of the type 
\begin{equation*}
F(v)=[(g_{ij}v^{i}v^{j})(h_{lk}v^{k}v^{l})]^{\frac{1}{4}}\,,
\end{equation*}%
where $g_{ij}$ and $h_{kl}$ are Lorentzian metrics of same signature type $%
(+,-,-,...,-)$, are relevant in physics in four spacetime dimensions, when
one describes the propagation of light in birefringent crystals, see for
example \cite{PerlickBook,Pfeifer:2016har,Punzi:2007di}.

Yet, here we will discuss the general $\left( n+1\right) $-dimensional case%
\textbf{. }We can always choose a basis the tangent spaces of the manifold
such that one of the bilinear forms $h$ or $g$ assumes its normal form, i.e.
it is locally diagonal with entries $g_{ij}=\eta _{ij}$.

\begin{proposition}
Consider $\mathbb{R}^{n+1}$ equipped with the Minkowski metric $\eta $ and
another bilinear form $h$ of Lorentzian signature. Let $\mathcal{T}\subset 
\mathbb{R}^{n+1}$ be the convex conic domain given by the intersection of
the future pointing timelike vectors of $\eta $ and $h$:%
\begin{equation*}
\mathcal{T}:=\left\{ v\in \mathbb{R}^{n+1}|\eta _{ij}v^{i}v^{j}>0,\
h_{ij}v^{i}v^{j}>0,v^{0}>0\right\} \subset \mathbb{R}^{n+1}
\end{equation*}%
and $F:\mathcal{T}\rightarrow \mathbb{R}^{+},$ the bimetric Finsler
structure defined by 
\begin{equation*}
F(v)=[(\eta _{ij}v^{i}v^{j})^{\frac{1}{4}}(h_{lk}v^{k}v^{l})]^{\frac{1}{4}%
}\,.
\end{equation*}%
Then, the fundamental inequality $F_{i}(v)w^{i}\geq F(w)$ is 
\begin{equation}
\frac{1}{2}\left( \frac{\eta _{ij}w^{i}v^{j}}{\eta _{kl}v^{k}v^{l}}+\frac{%
h_{ij}w^{i}v^{j}}{h_{kl}v^{k}v^{l}}\right) \geq \frac{F(w)}{F(v)},~\ \ \ \ \
\ \forall v,w\in \mathcal{T}\,.
\end{equation}
\end{proposition}

\begin{proof}
The Finsler structure is built from a fourth order polynomial $H(v):=(\eta
_{ij}v^{i}v^{j})(h_{lk}v^{k}v^{l})$ whose Hessian is given by 
\begin{equation*}
H_{ij}=2\eta _{ij}(h_{lk}v^{k}v^{l})+4(\eta _{ik}h_{jl}+\eta
_{jk}h_{il})v^{k}v^{l}+2h_{ij}(\eta _{kl}v^{k}v^{l});
\end{equation*}%
it was proven, see \cite{Pfeifer:2011tk,Pfeifer:2013gha,Raetzel:2010je},
that $H_{ij}$ has Lorentzian signature on $\mathcal{T}.$ Consequently, again
using Proposition \ref{prop:signHg}, $g_{ij}$ is of Lorentzian signature on $%
\mathcal{T}$ and the fundamental inequality $F_{i}(v)w^{i}\geq F(w)$ holds.
We easily calculate: 
\begin{equation*}
F_{i}(v)=\frac{1}{2}\frac{1}{H(v)^{\frac{3}{4}}}[\eta _{ij}v^{j}\
(h_{lk}v^{k}v^{l})+(\eta _{kl}v^{k}v^{l})\ h_{ij}v^{j}],
\end{equation*}%
and thus the inequality 
\begin{equation*}
\frac{1}{2}\frac{1}{H(v)^{\frac{3}{4}}}\left( \eta _{ij}w^{i}v^{j}\
(h_{lk}v^{k}v^{l})+(\eta _{kl}v^{k}v^{l})\ h_{ij}w^{i}v^{j}\right) \geq
H(w)^{\frac{1}{4}}
\end{equation*}%
holds. Both sides of the inequality can be multiplied with the positive
factor $H(v)^{-\frac{1}{4}}$ to obtain 
\begin{equation*}
\frac{1}{2}\left( \frac{\eta _{ij}w^{i}v^{j}}{\eta _{kl}v^{k}v^{l}}+\frac{%
h_{ij}w^{i}v^{j}}{h_{kl}v^{k}v^{l}}\right) \geq \left( \frac{H(w)}{H(v)}%
\right) ^{\frac{1}{4}}\,.
\end{equation*}
\end{proof}

\bigskip

As an explicit example, take $\mathbb{R}^{2}$ with $v=(v^{0},v^{1})$, $%
w=(w^{0},w^{1})$ and set $h_{kl}v^{k}v^{l}:=2(v^{0})^{2}-(v^{1})^{2}.$ As $%
\eta _{ij}v^{i}v^{j}=(v^{0})^{2}-(v^{1})^{2}$, we find for all $%
v^{0},v^{1},w^{0},w^{1}$ such that $(v^{0})^{2}>(v^{1})^{2}$ and $%
(w^{0})^{2}>(w^{1})^{2}:$%
\begin{equation}
\frac{1}{16}\left( \frac{v^{0}w^{0}-v^{1}w^{1}}{(v^{0})^{2}-(v^{1})^{2}}+%
\frac{2v^{0}w^{0}-v^{1}w^{1}}{2(v^{0})^{2}-(v^{1})^{2}}\right) ^{4}\geq 
\frac{((w^{0})^{2}-(w^{1})^{2})(2(w^{0})^{2}-(w^{1})^{2})}{%
((v^{0})^{2}-(v^{1})^{2})(2(v^{0})^{2}-(v^{1})^{2})}\,.
\end{equation}

\subsection{$\left( \protect\alpha ,\protect\beta \right) $-metrics. Kropina
metrics}

\label{Section_Kropina}

An intensively studied class of Finsler metrics studies are the so-called $%
(\alpha ,\beta )$-metrics, defined in terms of two building blocks: $\alpha
(v)=\sqrt{\eta (v,v)}$ and $\beta (v)=b_{i}v^{i}$, where $b_{i}$ are
components of an element $b\in V^{\ast }$. A general $(\alpha ,\beta )$%
-Finsler function can be expressed as 
\begin{equation}
F(\alpha ,\beta )=\alpha \phi (s),  \label{eq:ab}
\end{equation}%
where $\phi $ is a function of the 0-homogeneous variable $s=\frac{\beta }{%
\alpha }$.

Formally, the fundamental inequality (\ref{reverse_CS_coord_free}) of this
class of Finsler metrics takes the form 
\begin{equation}
\frac{\eta (v,w)}{\alpha (v)}\left[ \phi (s(v))-s(v)\phi ^{\prime }(s(v))%
\right] +\phi ^{\prime }(s(v))\beta (w)\geq \alpha (w)\phi (s(w))\,.
\label{fieqab}
\end{equation}

For the positive definite case, the conditions which need to be satisfied
such that \eqref{eq:ab} gives a well defined Finsler space are already
known, \cite[Lemma 1.1.2]{Bao}. For the indefinite case, the precise
conditions under which \eqref{eq:ab} defines a Lorentzian-Finsler structure
will be studied in a future work.

We now consider two examples of $(\alpha ,\beta )$-metrics: Randers metrics
- which lead to somewhat trivial results - and a first non-trivial example,
represented by Kropina metrics.

\bigskip

%Here is a first (albeit, trivial) example. \textit{Randers metrics }are
%defined by $F=\alpha +\beta ,$ i.e., $\phi (s)=s+1$ and are known to be
%Lorentzian if $\beta $ obeys: $\eta ^{ij}b_{i}b_{j}\in (0,1),$ see \cite%
%{Hohmann:2018rpp}. In this case, (\ref{fieqab}) gives: 
%\begin{equation*}
%\frac{\eta (v,w)}{\alpha (v)}+\beta (w)\geq \alpha (w)\left( \dfrac{\beta (w)%
%}{\alpha (w)}+1\right) ,
%\end{equation*}%
%which, after reducing the $\beta $-terms, leads to $\eta (v,w)\geq \alpha
%(v)\alpha (w),$ which is just the reverse Cauchy-Schwarz inequality for the
%undeformed metric $\eta ;$ so, in this case, no new inequalities appear.

%A first non-trivial example is represented by Kropina metrics, which we
%present below.

\subsubsection{Randers metrics}

The classical Randers metric, \cite{Randers}, is defined by the function $%
\phi (s)=1+s.$ They are known to be Lorentzian if $\beta $ obeys: $\eta
^{ij}b_{i}b_{j}\in (0,1),$ see \cite{Hohmann:2018rpp}. The fundamental
inequality \eqref{fieqab} becomes 
\begin{equation*}
\frac{\eta (v,w)}{\alpha (v)}+\beta (w)\geq \alpha (w)\left( \dfrac{\beta (w)%
}{\alpha (w)}+1\right) ,
\end{equation*}%
which, after canceling the $\beta $ terms, is nothing but the classical
reversed Cauchy-Schwarz inequality for the undeformed Lorentzian scalar
product $\eta $.

The same argument can be applied in the following cases:

a) If the Lorentzian scalar product $\eta $ is replaced by a positive
definite scalar product $h$. Then, the fundamental inequality 
\eqref{pos def
CS Finsler} yields the non-reversed Cauchy-Schwarz inequality for $h$.

b)\ For the deformed Randers metric $\tilde{F}=-\tilde{\alpha}+\beta =\tilde{%
\alpha}(-1+\tilde{s})$, where $\tilde{\alpha}=\sqrt{h(\dot{x},\dot{x})}$ is
constructed from a positive definite scalar product $h$, as in \cite%
{Caponio:2014gra}; $\tilde{F}$ defines a Lorentzian Finsler structure on $V,$
for $1-h^{ij}b_{i}b_{j}>0$. Replacing $\alpha $ with $\tilde{\alpha}$ and $%
\eta $ with $h$ in \eqref{fieqab}, one immediately deduces the classical
non-reversed Cauchy-Schwarz inequality for $h$. Thus, in this case, the
reversed fundamental inequality of a Lorentzian Finsler metric implies the
non-reversed Cauchy-Schwarz inequality for a positive definite scalar
product.

\subsubsection{Kropina metrics}

Another famous example of $(\alpha,\beta)$-metrics, are Kropina metrics,
which are applied for example in the context of Zermelo navigation, \cite%
{Javaloyes-Vitorio}. We will use a Kropina-type deformation of the Minkowski
metric $\eta$ in order to find out an inequality regarding $\eta$. They are
defined by choosing $\phi(s) = s^{-1}$ in \eqref{eq:ab}.

\begin{proposition}
Let $\mathcal{T}\subset \mathbb{R}^{n+1}$ be the convex conic domain 
\begin{equation*}
\mathcal{T}:=\left\{ v\in \mathbb{R}^{n+1}~|~\eta
_{ij}v^{i}v^{j}>0,v^{0}>0\right\} \subset \mathbb{R}^{n+1}\,
\end{equation*}%
and the smooth, 1-homogeneous function $F:\mathcal{T}\rightarrow \mathbb{R}%
^{+}$ be defined by 
\begin{equation*}
F(v)=\frac{\eta _{ij}v^{i}v^{j}}{v^{0}}=\dfrac{1}{v^{0}}%
[(v^{0})^{2}-(v^{1})^{2}-....-(v^{1})^{2}].
\end{equation*}%
Then, $F$ obeys the strict fundamental inequality (\ref%
{reverse_CS_coord_free}), which becomes:%
\begin{equation}
2\eta (v,w)\geq \frac{w^{0}}{v^{0}}\eta (v,v)+\dfrac{v^{0}}{w^{0}}\eta
(w,w),~\ \ \ \ \forall v,w\in \mathcal{T}.  \label{eq:kropineq}
\end{equation}
\end{proposition}

\begin{proof}
Let us rewrite $F$ as:%
\begin{equation*}
F(v)=v^{0}-\dfrac{\vec{v}\cdot \vec{v}}{v^{0}},
\end{equation*}%
where $v=(v^{0},\vec{v}),$ $\vec{v}=(v^{1},...,v^{n})$ and $\vec{v}\cdot 
\vec{v}=\delta _{\alpha \beta }v^{\alpha }v^{\beta }$ denotes the standard
Euclidean product on $\mathbb{R}^{n}.$ We easily get the derivatives of $F$
as: 
\begin{eqnarray}
F_{0}(v) &=&1+\dfrac{\vec{v}\cdot \vec{v}}{(v^{0})^{2}},~\ \ \ F_{\alpha
}(v)=-2\dfrac{\delta _{\alpha \beta }v^{\beta }}{v^{0}},~\ \ \alpha =1,...,n;
\label{derivs_F_1} \\
F_{00}(v) &=&-2\dfrac{\vec{v}\cdot \vec{v}}{(v^{0})^{3}},~\ \ F_{\alpha 0}=2%
\dfrac{\delta _{\alpha \beta }v^{\beta }}{(v^{0})^{2}},~\ \ \ F_{\alpha
\beta }=-2\dfrac{\delta _{\alpha \beta }}{v^{0}}.  \label{derivs_F_2}
\end{eqnarray}%
Now, we can prove that the matrix $g_{ij}(v)$ has Lorentzian signature for
all $v\in \mathcal{T}.$

First, we notice that $g_{v}(v,v)=g_{ij}(v)v^{i}v^{j}=F(v)^{2}>0$ on $%
\mathcal{T}$.

Second, we show that $g_{v}$ is negative semidefinite on the $g_{v}$%
-orthogonal complement of $v.$ This complement is defined by $%
g_{ij}(v)v^{j}w^{i}=0$, which is equivalent to: $F_{i}(v)w^{i}=0.$ Taking
into account the identity $g_{ij}(v)=FF_{ij}(v)+F_{i}(v)F_{j}(v)$, for $%
g_{v} $-orthogonal vectors $w,$ we can write:%
\begin{equation}
g_{ij}(v)w^{i}w^{j}=F(v)(F_{ij}(v)w^{i}w^{j}).  \label{g_rel_Kropina}
\end{equation}%
Taking into account (\ref{derivs_F_2}), we get: 
\begin{eqnarray*}
F_{ij}(v)w^{i}w^{j} &=&\dfrac{-2}{(v^{0})^{3}}\left[ (\vec{v}\cdot \vec{v}%
)\left( w^{0}\right) ^{2}-2\left( \vec{w}\cdot \vec{v}\right)
v^{0}w^{0}+\left( \vec{w}\cdot \vec{w}\right) (v^{0})^{2}\right] \\
&=&\dfrac{-2}{(v^{0})^{3}}(w^{0}\vec{v}-v^{0}\vec{w})\cdot (w^{0}\vec{v}%
-v^{0}\vec{w})\leq 0.
\end{eqnarray*}

Since $F(v)>0$ on $\mathcal{T},$ we obtain from (\ref{g_rel_Kropina}) that $%
g_{ij}(v)w^{i}w^{j}\leq 0$ on the $g_{v}$-orthogonal complement of $v$,
where $F_{ij}(v)w^{i}w^{j}=0$ implies $\vec{v}=\dfrac{v^{0}}{w^{0}}\vec{w};$
further, this leads to $(v^{0},\vec{v})=\dfrac{v^{0}}{w^{0}}(w^{0},\vec{w}),$
i.e., the vectors $v=(v^{0},\vec{v})$ and $w=\left( w^{0},\vec{w}\right) $
are collinear. Consequently, $g_{v}$ is Lorentzian and the resulting
fundamental inequality is strict.

The fundamental inequality can quickly be calculated either from (\ref%
{fieqab}) with $\phi =s^{-1}$, or with help of 
\begin{equation*}
F_{i}(v)=2\frac{\eta _{ij}v^{j}}{v^{0}}-\frac{\eta _{jk}v^{j}v^{k}}{%
(v^{0})^{2}}\delta _{i}^{0}
\end{equation*}%
which eventually leads to the desired inequality \eqref{eq:kropineq}.
\end{proof}

\section{\label{Aczel section}Acz\'{e}l inequality: generalizations and
refinements}

In the following, we will use a particular class of Lorentzian Finsler
metrics to obtain a generalization and some refinements of the Acz\'{e}l
inequality (\ref{1_1}).

Consider, on $\mathbb{R}^{n+1}\backslash \{0\},$ a smooth, positive definite
Finsler norm $\bar{F}$ and set:%
\begin{equation}
F(v):=\sqrt{\left( v^{0}\right) ^{2}-\bar{F}^{2}(\vec{v})},  \label{smooth F}
\end{equation}%
where $v=\left( v^{0},\vec{v}\right) $ belongs to the open conic subset of $%
\mathbb{R}^{n+1}\backslash \{0\}:$%
\begin{equation}
\mathcal{T}=\left\{ v=\left( v^{0},\vec{v}\right) \in \mathbb{R}%
^{n+1}~|v^{0}>\bar{F}(\vec{v})\right\} .  \label{T_set_Aczel}
\end{equation}%
Such functions $F$ are used in the context of stationary Lorentz Finsler
norms studied in \cite{Lammerzahl:2012kw,Stancarone} as well as in the
context of the Zermelo navigation \cite{Caponio:2014gra}.

As the function $\bar{F}$\ is convex, its epigraph $\mathcal{T}$ is convex;
it is also connected as it is the preimage of $(0,\infty )$ through the
continuous function $v\mapsto \left( v^{0}\right) -\bar{F}(\vec{v})$\textbf{%
. }Hence, $\mathcal{T}$ is a convex conic domain.

The metric tensor $g=\dfrac{1}{2}Hess(F^{2})$ is Lorentzian on $\mathcal{T},$
except for the half line $\left\{ \left( v^{0},0,0,....,0\right)
~|v^{0}>0\right\} ,$ where, due to the fact that $\bar{F}$ is only
continuous at $0\in \mathbb{R}^{n},$ $g_{v}$ is not guaranteed to exist.
Yet, the Acz\'{e}l inequality can still be extended to such metrics, as
follows.

\begin{proposition}
\textbf{(Finslerian Acz\'{e}l inequality)\label{Finslerian Aczel ineq}}%
\textit{: Let }$\bar{F}:\mathbb{R}^{n}\rightarrow \mathbb{R}$ be an
arbitrary Finsler norm on $\mathbb{R}^{n}$ and set%
\begin{equation}
\left\Vert v\right\Vert =\bar{F}(\vec{v}),~\ \ \ \ \ \bar{g}_{\vec{v}}(\vec{v%
},\vec{w}):=\dfrac{1}{2}\dfrac{\partial ^{2}\bar{F}^{2}}{\partial v^{\alpha
}\partial v^{\beta }}(\vec{v})v^{\alpha }w^{\beta }.  \label{alpha_norm}
\end{equation}%
Then, for all $v^{0},w^{0}>0$ and for all $\vec{v},\vec{w}\in \mathbb{R}%
^{n}\backslash \{0\}$ such that $\left( v^{0}\right) ^{2}-\left\Vert \vec{v}%
\right\Vert ^{2}>0,$ $\left( w^{0}\right) ^{2}-\left\Vert \vec{w}\right\Vert
^{2}>0$, there holds:%
\begin{equation}
\left[ v^{0}w^{0}-\vec{g}_{\vec{v}}(\vec{v},\vec{w})\right] ^{2}-[\left(
v^{0}\right) ^{2}-\left\Vert \vec{v}\right\Vert ^{2}][\left( w^{0}\right)
^{2}-\left\Vert \vec{w}\right\Vert ^{2}]\geq 0,  \label{Finslerian Aczel}
\end{equation}%
where equality takes place if and only if the vectors $v=(v^{0},\vec{v}),$ $%
\left( w^{0},\vec{w}\right) \in \mathbb{R}^{n+1}$ are collinear.
\end{proposition}

The proof is a direct one and relies on the following lemma:

\begin{lemma}
For all $v,w\in \mathbb{R}^{n+1},$ $v=(v^{0},\vec{v}),$ $w=(w^{0},\vec{w})$
with $\vec{w}\not=0,$ there holds:%
\begin{eqnarray}
&&\left[ v^{0}w^{0}-\vec{g}_{\vec{v}}(\vec{v},\vec{w})\right] ^{2}-[\left(
v^{0}\right) ^{2}-\left\Vert \vec{v}\right\Vert ^{2}][\left( w^{0}\right)
^{2}-\left\Vert \vec{w}\right\Vert ^{2}]=  \label{lemma} \\
&=&\left[ w^{0}\dfrac{\vec{g}_{\vec{v}}(\vec{v},\vec{w})}{\left\Vert \vec{w}%
\right\Vert }-v^{0}\left\Vert \vec{w}\right\Vert \right] ^{2}+\dfrac{\left(
w^{0}\right) ^{2}-\left\Vert \vec{w}\right\Vert ^{2}}{\left\Vert \vec{w}%
\right\Vert ^{2}}\left[ \left\Vert \vec{v}\right\Vert ^{2}\left\Vert \vec{w}%
\right\Vert ^{2}-\vec{g}_{\vec{v}}(\vec{v},\vec{w})\right] .  \notag
\end{eqnarray}
\end{lemma}

\begin{proof}
of the lemma:\ By direct computation, we find that both hand sides are
actually equal to: 
\begin{equation}
\left[ \left\Vert \vec{v}\right\Vert ^{2}\left( w^{0}\right) ^{2}+\left(
v^{0}\right) ^{2}\left\Vert \vec{w}\right\Vert ^{2}-2v^{0}w^{0}\vec{g}_{\vec{%
v}}(\vec{v},\vec{w})\right] +\left[ \vec{g}_{\vec{v}}(\vec{v},\vec{w}%
)-\left\Vert \vec{v}\right\Vert ^{2}\left\Vert \vec{w}\right\Vert ^{2}\right]
.  \label{lhs_Aczel}
\end{equation}
\end{proof}

\bigskip

\begin{proof}[Proof of Proposition \protect\ref{Finslerian Aczel ineq}]
Since $w\in \mathcal{T},$ we have $\left( w^{0}\right) ^{2}-\left\Vert \vec{w%
}\right\Vert ^{2}>0.$ Also, using the (non-reversed) Cauchy-Schwarz
inequality for $\bar{F},$ we find that the right hand side of (\ref{lemma})
is nonnegative; hence, the left hand side must also be nonnegative, which is
exactly (\ref{Finslerian Aczel}). \ Moreover, equality to zero can only
happen when both square brackets in the right hand side of (\ref{lemma})
vanish. This means, on one hand, $\vec{g}_{\vec{v}}(\vec{v},\vec{w}%
)-\left\Vert \vec{v}\right\Vert ^{2}\left\Vert \vec{w}\right\Vert ^{2}=0,$
which implies: $\vec{w}=\alpha \vec{v};$ the vanishing of the other bracket
then leads to $w^{0}=\alpha v^{0},$ i.e., the vectors $v,w\in \mathbb{R}%
^{n+1}$ are collinear.
\end{proof}

\bigskip

The above Lemma immediately yields two even more powerful results:

\begin{proposition}
\textbf{(Refinements of the Finslerian Acz\'{e}l inequality)}. Let $g_{v}$
be defined by a Finslerian norm $\bar{F}=\left\Vert \cdot \right\Vert $ on $%
\mathbb{R}^{n}$ as in (\ref{alpha_norm}). Then, for any $v^{0},w^{0}\in 
\mathbb{R}$ and any $\vec{v},\vec{w}\in \mathbb{R}^{n}\backslash \{0\}$%
\textbf{\ }such that\textbf{\ }$\left( v^{0}\right) ^{2}-\left\Vert \vec{v}%
\right\Vert ^{2}>0,$ $\left( w^{0}\right) ^{2}-\left\Vert \vec{w}\right\Vert
^{2}>0$\textbf{, }there hold the inequalities:

\begin{enumerate}
\item[\textit{(i)}] $\left[ v^{0}w^{0}-\vec{g}_{\vec{v}}(\vec{v},\vec{w})%
\right] ^{2}-[\left( v^{0}\right) ^{2}-\left\Vert \vec{v}\right\Vert
^{2}][\left( w^{0}\right) ^{2}-\left\Vert \vec{w}\right\Vert ^{2}]\geq 
\dfrac{\left( w^{0}\right) ^{2}-\left\Vert \vec{w}\right\Vert ^{2}}{%
\left\Vert \vec{w}\right\Vert ^{2}}(\left\Vert \vec{v}\right\Vert
^{2}\left\Vert \vec{w}\right\Vert ^{2}-\vec{g}_{\vec{v}}(\vec{v},\vec{w}));$

\item[\textit{(ii)}] $\left[ v^{0}w^{0}-\vec{g}_{\vec{v}}(\vec{v},\vec{w})%
\right] ^{2}-[\left( v^{0}\right) ^{2}-\left\Vert \vec{v}\right\Vert
^{2}][\left( w^{0}\right) ^{2}-\left\Vert \vec{w}\right\Vert ^{2}]\geq
\lbrack w^{0}\dfrac{\vec{g}_{\vec{v}}(\vec{v},\vec{w})}{\left\Vert \vec{w}%
\right\Vert }-v^{0}\left\Vert \vec{w}\right\Vert ]^{2}.$
\end{enumerate}
\end{proposition}

\bigskip

\textbf{Note. }A somewhat similar statement can be obtained by considering,
on the entire space $V\simeq \mathbb{R}^{n+1},$ a positive definite Finsler
norm $\hat{F}$ and a 1-form $\omega =\omega _{i}dx^{i}.$ Then, the mapping $%
v\mapsto F(v)=\sqrt{\omega ^{2}(v)-\hat{F}^{2}(v)}$ defines (see Theorem 4.1
in \cite{Javaloyes2019}), a smooth, nondegenerate Lorentz-Finsler norm on
the convex conic domain $\mathcal{T}=\mathcal{\{}v\in V~|~\omega (v)>\hat{F}%
(v)\}~\subset V\backslash \{0\}.$ The reverse Cauchy-Schwarz inequality then
reads:%
\begin{equation}
\lbrack \omega (v)\omega (w)-\hat{g}_{v}(v,w)]^{2}\geq \lbrack \omega
^{2}(v)-\hat{F}^{2}(v)][\omega ^{2}(w)-\hat{F}^{2}(w)].  \label{Aczel_gen_2}
\end{equation}

\bigskip

\textbf{Acknowledgements} C.P. was supported by the Estonian Ministry for
Education and Science through the Personal Research Funding Grants PSG489,
as well as the European Regional Development Fund through the Center of
Excellence TK133 \textquotedblleft The Dark Side of the
Universe\textquotedblright . N.V. was supported by a local grant of
Transilvania University\textbf{.}

We would like to express our gratitude to the anonymous MIA reviewer, whose
suggestions led us to extending our results and enlarging the presented
classes of examples.

\begin{appendices}
\section{The signature of $m$-th root metrics}

\label{app:mthroot} Among the Lorentz-Finsler norms in Section~\ref{sec:ex},
we encountered $m$-th root metrics, which are functions of the type: 
\begin{equation*}
F:\mathcal{T}\rightarrow \mathbb{R},\quad F(v)=H(v)^{\tfrac{1}{m}},\quad 
\text{with}\quad H(v):=a_{i_{1}...i_{m}}v^{i_{1}}...v^{i_{m}}\,,
\end{equation*}%
where $a_{i_{1}...i_{m}}$ are constants and $\mathcal{T}$ is a convex conic
domain in\textbf{\ }$\mathbb{R}^{n+1}$\textbf{\ }where\textbf{\ }$H(v)>0$.
Here, $m>2$ is fixed.

To determine the signature of $g_{ij}=\frac{1}{2}\frac{\partial F^{2}}{%
\partial v^{i}\partial v^{j}}$ in some of our examples, we relate it to the
signature of the Hessian $H_{ij}:=\frac{\partial H}{\partial v^{i}\partial
v^{j}}$ of $H$. By direct calculation, we get (see also  \cite{Brinzei-projective}, \cite{Pfeifer:2011tk}%): 
\begin{equation}
H_{ij}=mF^{m-2}[g_{ij}+(m-2)F_{i}F_{j}].  \label{P_ij}
\end{equation}%
The following result will greatly help us simplify calculations in concrete
examples.

\begin{proposition}
\label{prop:signHg} If, for a vector $v\in \mathcal{T},$ the matrix $%
H_{ij}(v)$ has Lorentzian signature $\left( +,-,-,..,-\right) ,$ then $%
g_{ij}(v)$ has also Lorentzian signature.
\end{proposition}

\begin{proof}
Assume that $H_{ij}(v)$ has index $n$; therefore, there exists an $n$
-dimensional subspace of $\mathbb{R}^{n+1}$ where it is negative definite.
Pick any vector $w\in \mathcal{T}$ with the property $H_{ij}(v)w^{i}w^{j}<0$%
. From (\ref{P_ij}) we find: $%
g_{ij}(v)w^{i}w^{j}+(m-2)(l_{i}(v)w^{j})^{2}<0. $ Since $m>2,$ this implies $%
g_{ij}(v)w^{i}w^{j}<0,$ that is, $(g_{ij})$ is also negative definite on at
least the same $n$-dimensional subspace.

But,\ $g_{ij}(v)v^{i}v^{j}=F^{2}(v)>0,$ which means that $g_{ij}$ cannot be
negative (semi-)definite on the whole $V.$ Consequently, it must be
Lorentzian.
\end{proof}
\end{appendices}


\begin{thebibliography}{99}
\bibitem{Aazami:2014ata} A.~B.~Aazami and M.~A.~Javaloyes, \textit{Penrose
singularity theorem in a Finsler spacetime}, Class. Quant. Grav. \textbf{33}
(2016) no.2, 025003 doi:10.1088/0264-9381/33/2/025003.

\bibitem{AC} J. Acz\'{e}l, \textit{Some general methods in the theory of
functional equations in one variable, New applications of functional
equations (Russian)}, Uspehi Mat. Nauk (N.S.) \textbf{11}, 69(3) (1956),
3-68.

%\bibitem{Asanov book} G.S. Asanov, \textit{Finsler geometry, relativity and
%gauge theories,} Reidel, Dordrecht, 1985.

\bibitem{Asanov1980} G.S. Asanov, \textit{Finsler space with an algebraic
metric defined by a field of frames}, J Math Sci 13, 588--600 (1980).

\bibitem{Bao} Bao, D.; Chern, S.S.; Shen, Z. \textit{An Introduction to
Finsler-Riemann Geometry}, Springer, New York, 2000.

\bibitem{Bel} R. Bellman, \textit{On an inequality concerning an indefinite
form}, Amer. Math. Monthly \textbf{63}(1956), 108-109.

\bibitem{BeemEhrlich} J. K. Beem, P. E. Ehrlich, and K. L. Easley, \textit{%
Global Lorentzian geometry}, volume 202 of Monographs and Textbooks in Pure
and Applied Mathematics. Marcel Dekker, Inc., New York, second edition, 1996.

\bibitem{Bernal:2020bul} A.~Bernal, M.~A.~Javaloyes and M.~Sanchez, \textit{%
Foundations of Finsler spacetimes from the Observers' Viewpoint}, Universe 
\textbf{6} (2020) no.4, 55 doi:10.3390/universe6040055. 
%2 citations counted in INSPIRE as of 16 Jun 2020

\bibitem{Brinzei-projective} N. Brinzei, \textit{Projective relations for
m-th root metric spaces}, Journal of the Calcutta Mathematical Society
5(1-2), 21-35 (2009).

\bibitem{Stancarone} E. Caponio, G. Stancarone, \textit{On Finsler
spacetimes with a timelike Killing vector field}, Class.Quant.Grav. 35
(2018) 8, 085007.

%\cite{Caponio:2014gra}

\bibitem{Caponio:2014gra} E.~Caponio, M.~A.~Javaloyes and M.~S\'{a}nchez, 
\textit{Wind Finslerian structures: from Zermelo's navigation to the
causality of spacetimes}, [arXiv:1407.5494 [math.DG]]. 
%15 citations counted in INSPIRE as of 22 Oct 2020

\bibitem{cosmo-Berwald} M. Hohmann, C. Pfeifer, N. Voicu, \textit{%
Cosmological Finsler spacetimes, }Universe 6(5), 65 (2020).

\bibitem{Hohmann:2018rpp} M.~Hohmann, C.~Pfeifer and N.~Voicu, \textit{%
Finsler gravity action from variational completion}, Phys. Rev. D \textbf{100%
} (2019) no.6, 064035 doi:10.1103/PhysRevD.100.064035.

\bibitem{Javaloyes2019} M.A. Javaloyes, M. Sanchez, \textit{On the
definition and examples of cones and Finsler spacetimes}, Revista de la Real
Academia de Ciencias Exactas, F\'{\i}sicas y Naturales. Serie A. Matem\'{a}%
ticas \textbf{114}, 30 (2020).

\bibitem{Javaloyes2014} M.A. Javaloyes, M. Sanchez, \textit{On the
definition and examples of Finsler metrics, }Ann. Sc. Norm. Super. Pisa Cl.
Sci. 5, XIII, 813-858 (2014).

\bibitem{Javaloyes-Vitorio} M. A. Javaloyes and H. Vitorio, \textit{Some
properties of Zermelo navigation in pseudo-Finsler metrics under an
arbitrary wind}, Houston Journal of Mathematics, 44 (4):1147-1179 (2018).

\bibitem{Kropina} V.K. Kropina, \textit{On projective two-dimensional
Finsler spaces with a special metric}, Trudy Sem. Vektor. Tenzor. Anal., 
\textbf{11}. 277 - 292 (1961).

\bibitem{Lammerzahl:2012kw} C.~Lammerzahl, V.~Perlick and W.~Hasse, \textit{%
Observable effects in a class of spherically symmetric static Finsler
spacetimes}, Phys. Rev. D \textbf{86}, 104042 (2012)
doi:10.1103/PhysRevD.86.104042.

\bibitem{Min-Pal} N. Minculete, R. P\u{a}lt\u{a}nea \textit{Improved
estimates for the triangle inequality}, J. Ineq. Appl. (2017)(1), 1-12.

\bibitem{Minguzzi2019} E. Minguzzi, \textit{Lorentzian causality theory},
Living Reviews in Relativity (2019) 22:3.

\bibitem{Minguzzi2014} E. Minguzzi, \textit{Light Cones in Finsler Spacetime}%
, Communications in Mathematical Physics 334, 1529--1551(2015).

\bibitem{Minguzzi:2013sxa} E.~Minguzzi, \textit{Convex neighborhoods for
Lipschitz connections and sprays}, Mon. Math. \textbf{177} (2015), 569-625
doi:10.1007/s00605-014-0699-y.

\bibitem{Mitrinovic1970} D. S. Mitrinovi\'{c}, \textit{Analytic Inequalities}%
, Springer, Berlin-Heidelberg-New York, 1970.

\bibitem{Mitrinovic1993} D. S. Mitrinovi\'{c}, J. Pe\v{c}ari\'{c}, A. M.
Fink, \textit{Classical and New Inequalities in Analysis}, Springer,
Dordrecht 1993.

\bibitem{ONeil} B. O'Neill, \textit{Semi-Riemannian Geometry With
Applications to Relativity}, Academic Press, San Diego-London, 1983.

\bibitem{PerlickBook} V.~Perlick, \textit{Ray Optics, Fermat's Principle,
and Applications to General Relativity}, No. 61 in Lecture Notes in Physics.
Springer, (2000).

\bibitem{Pfeifer:2011tk} C.~Pfeifer and M.~N.~Wohlfarth, \textit{Causal
structure and electrodynamics on Finsler spacetimes}, Phys. Rev. D \textbf{84%
} (2011), 044039 doi:10.1103/PhysRevD.84.044039.

\bibitem{Pfeifer:2013gha} C.~Pfeifer, \textit{The Finsler spacetime
framework: backgrounds for physics beyond metric geometry},
DESY-THESIS-2013-049. %4 citations counted in INSPIRE as of 25 May 2020

\bibitem{Pfeifer:2016har} C.~Pfeifer and D.~Siemssen, \textit{%
Electromagnetic potential in pre-metric electrodynamics: Causal structure,
propagators and quantization}, Phys. Rev. D \textbf{93} (2016) no.10, 105046
doi:10.1103/PhysRevD.93.105046. 
%13 citations counted in INSPIRE as of 25 May 2020

\bibitem{Pfeifer:2019wus} C.~Pfeifer, \textit{Finsler spacetime geometry in
Physics},\ Int. J. Geom. Meth. Mod. Phys. \textbf{16}, no.supp02, 1941004
(2019) doi:10.1142/S0219887819410044.

\bibitem{Pop} T. Popoviciu, \textit{On an inequality} (Romanian), Gaz. Mat.
Fiz. Ser. A 8(64) (1959), 451-461.

\bibitem{Punzi:2007di} R.~Punzi, M.~N.~Wohlfarth and F.~P.~Schuller, \textit{%
Propagation of light in area metric backgrounds},\ Class. Quant. Grav. 
\textbf{26} (2009), 035024 doi:10.1088/0264-9381/26/3/035024.

\bibitem{Randers} G.~Randers, \textit{On an Asymmetrical Metric in the
Four-Space of General Relativity},\ Phys. Rev. \textbf{59} (1941), 195-199
doi:10.1103/PhysRev.59.195.

\bibitem{Raetzel:2010je} D.~Raetzel, S.~Rivera and F.~P.~Schuller, \textit{%
Geometry of physical dispersion relations},\ Phys. Rev. D \textbf{83}
(2011), 044047 doi:10.1103/PhysRevD.83.044047.
\end{thebibliography}
\end{document}